\documentclass[12pt,twoside]{amsart}
\usepackage{amsmath, amsthm, amscd, amsfonts, amssymb, graphicx}
\usepackage[bookmarksnumbered, plainpages]{hyperref}

\newcommand{\lmd}{\lambda}

\newcommand{\og}{\omega}

\newcommand{\tr}{\text{tr}}
\newcommand{\therm}{\text{Herm}}
\newcommand{\tend}{\text{End}}
\newcommand{\lag}{\langle}
\newcommand{\rag}{\rangle}

\newtheorem{thm}{Theorem}[section]
\newtheorem{cor}[thm]{Corollary}
\newtheorem{lem}[thm]{Lemma}
\newtheorem{prop}[thm]{Proposition}
\newtheorem{defn}[thm]{Definition}

\numberwithin{equation}{section}

\begin{document}

\title{\bf Higgs bundles over non-compact Gauduchon manifolds}
\author{Chuanjing Zhang, Pan Zhang and Xi Zhang}

\thanks{{\scriptsize
\hskip -0.4 true cm \textit{2010 Mathematics Subject Classification:}
53C07; 14J60; 32Q15.
}}

\maketitle

\begin{abstract}
In this paper, we prove a generalized Donaldson-Uhlenbeck-Yau theorem on Higgs bundles over a class of non-compact Gauduchon manifolds.
\end{abstract}

\vskip 0.2 true cm


\pagestyle{myheadings}
\markboth{\rightline {\scriptsize C. Zhang et al.}}
         {\leftline{\scriptsize Higgs bundles over non-compact Gauduchon manifolds}}

\bigskip
\bigskip


\section{ Introduction}

Let $X$ be a  complex manifold of dimension $n$ and $g$ a Hermitian metric with associated K\"{a}hler form $\omega$. The metirc $g$ is called Gauduchon if $\omega$ satisfies $\partial \bar{\partial} \omega^{n-1}=0$.  A Higgs bundle $(E , \overline{\partial }_{E}, \theta )$ over $X$ is a holomorphic bundle $(E , \overline{\partial }_{E})$ coupled with a Higgs field $\theta \in \Omega_X^{1,0}(\mathrm{End}(E))$
such that $\overline{\partial}_{E}\theta =0$ and $\theta \wedge \theta =0$.
Higgs bundles were introduced by Hitchin (\cite{HIT}) in his study of the self duality
equations. They have  rich structures and play an important role in many  areas
including gauge theory, K\"ahler and hyperk\"ahler geometry, group representations
and nonabelian Hodge theory. Let $H$ be a Hermitian metric on the bundle $E$, we consider the Hitchin-Simpson connection
 \begin{equation*}
 \overline{\partial}_{\theta}:=\overline{\partial}_{E}+\theta , \quad D_{H,  \theta }^{1, 0}:=D_{H }^{1, 0}  +\theta^{*_H}, \quad D_{H,  \theta }= \overline{\partial}_{\theta}+ D_{H,  \theta }^{1, 0},
 \end{equation*}
where $D_{H}$ is the Chern connection of $(E,\overline{\partial }_{E},H)$ and $\theta^{*_H}$ is the adjoint of $\theta $ with respect to the metric $H$.
The curvature of this connection is
\begin{equation*}
F_{H,\theta}=F_H+[\theta,\theta^{*_H}]+\partial_H\theta+\bar{\partial}_E\theta^{*_H},
\end{equation*}
where $F_H$ is the curvature of $D_{H}$ and $\partial_H$ is the $(1, 0)$-part of $D_{H}$.  $H$  is said to be
a Hermitian-Einstein metric on Higgs bundle $(E, \overline{\partial }_{E}, \theta  )$ if the
curvature  of the Hitchin-Simpson connection
  satisfies the Einstein condition, i.e.
\begin{equation*}
\sqrt{-1}\Lambda_{\omega} (F_{H} +[\theta , \theta^{*_H}])
=\lambda\cdot \mathrm{Id}_{E},
\end{equation*}
where  $\Lambda_{\omega }$ denotes the contraction with  $\omega $, and $\lambda $ is a constant.

When the base space $(X, \omega)$ is a compact K\"ahler manifold,
the stability of Higgs bundles, in the sense of Mumford-Takemoto, was a well established concept. Hitchin (\cite{HIT}) and Simpson (\cite{SIM}, \cite{SIM2}) obtained a Higgs
bundle version of the Donaldson-Uhlenbeck-Yau theorem (\cite{NS65}, \cite{DON85}, \cite{UY86}), i.e. they proved that a Higgs bundle
admits the Hermitian-Einstein metric if and only if it's Higgs poly-stable. Simpson (\cite{SIM}) also considered some non-compact K\"ahler manifolds case, he introduced the concept of analytic stability for Higgs bundle, and proved that the analytic stability implies the existence of Hermitian-Einstein metric. There
are  many other interesting and important works related (\cite{ag,Bi,bis,br,DW,JZ, LN2, LZ, LZZ, LZZ2, Mo1, Mo2, Mo3,NR01,WZ}, etc.). The non-K\"ahler  case is also very interesting. The Donaldson-Uhlenbeck-Yau theorem is  valid for compact Gauduchon manifolds (see \cite{Bu, LY, LT,LT95}).

In this paper, we want to study the non-compact and non-K\"ahler case. In the following, we always suppose that $(X, g)$ is a Gauduchon manifold unless otherwise stated.
By \cite{SIM}, we will make the following three assumptions:

\medskip

{\bf Assumption 1.} $(X, g)$ has finite volume.

{\bf Assumption 2.} There exists a non-negative exhaustion function $\phi$ with $\sqrt{-1}\Lambda_{\omega }\partial \bar{\partial} \phi$ bounded.

{\bf Assumption 3.} There is an increasing function $a:[0,+\infty)\rightarrow [0,+\infty)$ with $a(0)=0$ and $a(x)=x$ for $x>1$, such that if $f$ is a bounded positive function on $X$ with $\sqrt{-1}\Lambda_{\omega } \partial \bar{\partial} f\geq -B$ then
$$\sup_X|f|\leq C(B)a(\int_X|f|\frac{\omega^{n}}{n!}).$$
Furthermore, if $\sqrt{-1}\Lambda_{\omega } \partial \bar{\partial} f\geq 0$, then $\sqrt{-1}\Lambda_{\omega } \partial \bar{\partial} f=0$.

\medskip

We fix a background metric $K$ in the bundle $E$, and suppose that
\begin{equation*}
\sup_X|\Lambda_{\omega} F_{K,\theta}|_K<+\infty.
\end{equation*}
Define the analytic degree of $E$ to be the real number
\begin{equation*}
\textmd{deg}_{\omega}(E,K)=\sqrt{-1}\int_X\textmd{tr}(\Lambda_{\omega} F_{K,\theta})\frac{\omega^{n}}{n!}.
\end{equation*}
According to the Chern-Weil formula with respect to the metric $K$ (Lemma 3.2 in \cite{SIM}), we can define the analytic degree of any saturated sub-Higgs sheaf $V$ of $(E , \overline{\partial }_{E}, \theta )$ by
\begin{equation} \label{cw}
\textmd{deg}_{\omega}(V,K)=\int_X \sqrt{-1}\textmd{tr}(\pi \Lambda_{\omega} F_{K,\theta})-|\overline{\partial }_{\theta}\pi |_{K}^{2}\frac{\omega^{n}}{n!},
\end{equation}
where $\pi $ denotes the projection onto $V$ with respect to the metric $K$. Following \cite{SIM}, we say that the Higgs bundle $(E , \overline{\partial }_{E}, \theta )$  is $K$-analytic stable (semi-stable) if for every proper saturated sub-Higgs sheaf $V\subset E$,
\begin{equation*}
\frac{\textmd{deg}_{\omega}(V,K)}{\textmd{rank}(V)}<(\leq)\frac{\textmd{deg}_{\omega}(E,K)}{\textmd{rank}(E)}.
\end{equation*}

In this paper, we will show that, under some assumptions on the base space $(X, g)$, the analytic stability implies the existence of Hermitian-Einstein metric on $(E , \overline{\partial }_{E}, \theta )$, i.e. we obtain the following Donaldson-Uhlenbeck-Yau type theorem.

\begin{thm} \label{theorem1}
Let $(X, g)$ be a non-compact Gauduchon manifold satisfying the Assumptions 1,2,3, and $|\mathrm{d}\omega^{n-1}|_{g}\in L^2(X)$,  $(E,\bar{\partial}_E,\theta)$ be a Higgs bundle over $X$ with a Hermitian metric $K$ satisfying $\sup_X|\Lambda_{\omega} F_{K,\theta}|_{K}<+\infty$. If $(E,\bar{\partial}_E,\theta)$ is $K$-analytic stable, then there exists a Hermitian metric $H$ with $\overline{\partial }_{\theta} (\log K^{-1}H)\in L^{2}$,  $H$ and $K$ are mutually bounded,  such that
\begin{equation*}
\sqrt{-1}\Lambda_{\omega} (F_{H} +[\theta , \theta ^{*_H}])
=\lambda_{K, \omega}\cdot \mathrm{Id}_{E},
\end{equation*}
where the constant $\lambda_{K, \omega }=\frac{\mathrm{deg}_{\omega}(E,K)}{\mathrm{rank}(E) \mathrm{Vol}(X, g)}$.
\end{thm}

 From the Chern-Weil formula (\ref{cw}), it is easy to see that the existence of Hermitian-Einstein metric $H$ implies $(E,\bar{\partial}_E,\theta)$ is $H$-analytic poly-stable. Our result is slightly better than that in \cite{SIM}, where Simpson only obtained a Hermitian metric with vanishing trace-free curvature. The reason is that, in Section $4$, we can solve the following Poisson equation
\begin{equation}\label{po1}
-2\sqrt{-1}\Lambda_{\omega }  \bar{\partial} \partial f =\psi
\end{equation}
on the non-K\"ahler and non-compact manifold $(X, g)$ when $\int_{X}\psi\frac{\omega^{n}}{n!}=0$. In \cite{SIM}, Simpson used Donaldson's heat flow method to attack the existence problem of the Hermitian-Einstein metrics on Higgs bundles, and his proof relies on the properties of the Donaldson functional. However, the Donaldson functional is not well-defined when $g$ is only Gauduchon. So Simpson's argument is not applicable in our situation directly. In this paper, we follow the argument of Uhlenbeck-Yau in \cite{UY86}, where they used the continuity method and their argument is more natural. We first solve  the following perturbed equation on $(X, g)$:
\begin{equation} \label{eq}
\sqrt{-1}\Lambda_{\omega } (F_{H}+[\theta,\theta^{*_H}])-\lambda \cdot \textmd{Id}_E+\varepsilon \log (K^{-1}H)=0.
\end{equation}
The above perturbed equation can be solved by using the fact that the elliptic operators are Fredholm if the base manifold is compact. Generally speaking, this fact is not true in the non-compact case, which means we can not directly apply this method to solve the perturbed equation on the non-compact manifold. To fix this, we combine the method of heat flow and the method of exhaustion to solve the perturbed equation on $(X, g)$ for any $0< \varepsilon \leq 1$, see Section $5$ for details. For simplicity, we  set
\begin{equation} \Phi(H, \theta)=\sqrt{-1}\Lambda_{\omega } (F_{H}+[\theta,\theta^{*_H}])-\lambda_{K , \omega } \cdot \textmd{Id}_E. \end{equation}
Under the assumptions as that in Theorem \ref{theorem1}, we can prove the following identity:
\begin{equation}\label{eq33}
\int_X \tr(\Phi(K,\theta) s)+\lag\Psi(s)(\overline{\partial}_{\theta}s), \overline{\partial}_{\theta}s\rag_{K}\frac{\og^n}{n!}=\int_X \tr(\Phi(H,\theta) s),
\end{equation}
where $s=\log (K^{-1}H)$ and
\begin{equation}\label{eq3301}
\Psi(x,y)=
\begin{cases}
&\frac{e^{y-x}-1}{y-x},\ \ \ x\neq y;\\
&\ \ \ \  1,\ \ \ \ \ \  x=y.
\end{cases}
\end{equation}
By the above identity (\ref{eq33}) and Uhlenbeck-Yau's result (\cite{UY86}), that $L_{1}^{2}$
weakly holomorphic sub-bundles define coherent sub-sheaves, we can obtain the existence result of Hermitian-Einstein metric by using the continuity method. It should be pointed out that application of the identity (\ref{eq33}) plays a key role  in our argument (see Section 6), which is slightly different with that in \cite{UY86} (or \cite{Bu, LY, LT}).

In the end of this paper, we also study the semi-stable case. A Higgs bundle is said to be admitting an approximate Hermitian-Einstein structure, if for every $\delta >0$, there exists a Hermitian metric $H$ such that
\begin{equation*}
\sup_X|\sqrt{-1}\Lambda_{\omega }(F_H+[\theta,\theta^{*_H}])-\lambda_{K, \omega } \cdot \textmd{Id}_E|_H<\delta.
\end{equation*}
This notion was firstly introduced by Kobayashi(\cite{Kobayashi})  in  holomorphic vector bundles (i.e. $\theta =0$). He proved that over projective manifolds, a semi-stable holomorphic vector bundle  must admit an approximate Hermitian-Einstein structure. In \cite{LZ},  Li and  the third author proved this result is valid for Higgs bundles over compact K\"ahler manifolds. There are also some other interesting works related, see references \cite{bruzzo, Cardona, Jacob1, NZ} for details. In this paper, we obtain an existence result of approximate Hermitian-Einstein structures on  analytic semi-stable Higgs bundles over a class of non-compact Gauduchon  manifolds. In fact, we prove that:

\begin{thm} \label{theorem2}
Under the  same assumptions as that in Theorem \ref{theorem1}, if the Higgs bundle $(E,\bar{\partial}_E,\theta)$ is $K$-analytic semi-stable, then there must exist an approximate Hermitian-Einstein structure, i.e. for every $\delta >0$, there exists a Hermitian metric $H$ with $H$ and $K$  mutually bounded,  such that
\begin{equation*}\sup_X|\sqrt{-1}\Lambda_{\omega }(F_H+[\theta,\theta^{*_H}])-\lambda_{K, \omega } \cdot \mathrm{Id}_E|_H<\delta .\end{equation*}
\end{thm}

This paper is organized as follows. In Section 2,  we give some estimates and
preliminaries which will be used in the proof of Theorems \ref{theorem1} and \ref{theorem2}. At the end of Section 2, we prove the identity (\ref{eq33}). In Section 3, we get the long-time existence result
of the related heat flow. In Section 4, we consider the Poisson equation (\ref{po1}) on some non-compact Gauduchon manifolds. In Section 5, we solve
the perturbed equation (\ref{eq}). In Section 6,  we complete the proof of Theorems \ref{theorem1} and \ref{theorem2}.

\section{Preliminary results}

 Let $(M, g)$ be an $n$-dimensional  Hermitian manifold. Let $(E,\bar{\partial}_E,\theta)$ be
a rank $r$ Higgs bundle over $M$ and $H_{0}$ be a Hermitian metric on $E$.
We consider the following heat flow.
\begin{equation} \label{Flow}
H^{-1}\frac{\partial H}{\partial t}=-2(\sqrt{-1}\Lambda_{\omega } (F_{H}+[\theta,\theta^{*_H}])-\lambda \cdot \textmd{Id}_E+\varepsilon \log (H_0^{-1}H)),
\end{equation}
where $H(t)$ is  a family of Hermitian metrics  on $E$ and $\varepsilon$ is a nonnegative constant.
Choosing local
complex coordinates $\{ z^{i}\}_{i=1}^{n}$ on $M$, then $\omega =\sqrt{-1} g_{i \bar{j }} dz^{i} \wedge d\overline{z}^{j}$.  We
define the complex Laplace operator for functions
$$
\widetilde{\Delta}f=-2\sqrt{-1}\Lambda_{\omega } \bar{\partial }\partial
f =2g^{i \bar{j }}\frac{\partial ^{2 }f}{\partial z^{i }\partial
\bar{z}^{j }},
$$
where $(g^{i \bar{j }})$ is the inverse matrix of the metric
matrix $(g_{i \bar{j }})$.
As usual, we denote the Beltrami-Laplcaian operator by $\Delta$. It is well known that the difference of the two Laplacians is given by a first order differential operator as follows
\begin{equation*} \label{laplacian}
(\widetilde{\Delta}-\Delta)f=\langle V,\nabla f\rangle_g,
\end{equation*}
where $V$ is a well-defined vector field on $M$.

\begin{prop}  \label{P1}
Let $H(t)$ be a solution of the flow \eqref{Flow}, then
\begin{equation}\label{trace}
(\frac{\partial}{\partial t}-\widetilde{\Delta})\{e^{2\varepsilon t}\cdot \mathrm{tr} (\sqrt{-1}\Lambda_{\omega }(F_H+[\theta,\theta^{\ast_{H}}])-\lambda \cdot \mathrm{Id}_E+\varepsilon \log h)\}=0
\end{equation}
and
\begin{equation} \label{dec}
(\frac{\partial}{\partial t}-\widetilde{\Delta})|\sqrt{-1}\Lambda_{\omega }(F_H+[\theta,\theta^{\ast_{H}}])-\lambda \cdot \mathrm{Id}_E+\varepsilon \log h|^2_{H}\leq 0.
\end{equation}
\end{prop}

\begin{proof}
For simplicity, we denote $\sqrt{-1}\Lambda_{\omega }(F_H+[\theta,\theta^{\ast_{H}}])-\lambda \cdot \textmd{Id}_E+\varepsilon \log h=\Phi_{\varepsilon}$. By calculating directly, we have
\begin{equation}\label{trace1}
\frac{\partial}{\partial t}\Phi_{\varepsilon}
=\sqrt{-1}\Lambda_{\omega }\{\bar{\partial}_E(\partial_H(h^{-1}\frac{\partial h}{\partial t}))+[\theta,[\theta^{*_H},h^{-1}\frac{\partial h}{\partial t}]]\}
+\varepsilon \frac{\partial}{\partial t}(\log h),
\end{equation}
and
\begin{equation*}
\begin{split}
\widetilde{\Delta}|\Phi_{\varepsilon}|^2_H&=-2\sqrt{-1}\Lambda_{\omega } \bar{\partial}\partial \textmd{tr}\{\Phi_{\varepsilon} H^{-1}\bar{\Phi}_{\varepsilon}^{t}H\}\\
&=-2\sqrt{-1}\Lambda_{\omega } \bar{\partial} \textmd{tr}\{\partial \Phi_{\varepsilon} H^{-1}\bar{\Phi}_{\varepsilon}^{t}H-\Phi_{\varepsilon} H^{-1}\partial H H^{-1}\bar{\Phi}_{\varepsilon}^{t}H\\
&~~~~+\Phi H^{-1}\overline{{\bar{\partial}\Phi_{\varepsilon}}}^{t}H+\Phi_{\varepsilon} H^{-1}\bar{\Phi}_{\varepsilon}^{t}H H^{-1}\partial H\}\\
&=2\textmd{Re} \langle -2\sqrt{-1}\Lambda_{\omega } \bar{\partial}_E\partial_H{\Phi_{\varepsilon}},\Phi_{\varepsilon}\rangle_H
+2|\partial_H\Phi_{\varepsilon}|^2_H+2|\bar{\partial}_E\Phi_{\varepsilon}|^2_H.
\end{split}
\end{equation*}
From (\ref{trace1}), it is easy to conclude that
\begin{equation}\label{trace2}
(\frac{\partial}{\partial t}-\widetilde{\Delta}) \tr \Phi_{\varepsilon}=-2\varepsilon \tr \Phi_{\varepsilon}.
\end{equation}
Then, (\ref{trace2}) implies (\ref{trace}).

From \protect{\cite[p. 237]{LT95}}, we can choose an open dense subset $W\subset M\times [0, T_{0}]$ satisfying at each $(x_{0}, t_{0})\in W$ there exist an open neighborhood $U$ of $(x_{0}, t_{0})$, a local unitary basis $\{e_i\}_{i=1}^r$ with respect to $H$ and functions $\{\lambda_i\in C^{\infty}(U,\mathbb{R})\}_{i=1}^r$ such that
$$h(y, t)=\sum_{i=1}^re^{\lambda_i(y, t)}e_i(y,t)\otimes e^i(y, t)$$
for all $(y, t)\in U$, where $\{e^i\}_{i=1}^r$ is the corresponding dual basis. Then
 we get
$$\frac{\partial}{\partial t}(\log h)=\sum_{i=1}^r(\frac{\textmd{d} \lambda_i}{\textmd{d} t})e_i\otimes e^i+\sum_{i\neq j}(\lambda_j-\lambda_i)\alpha_{ji}e_i\otimes e^j,$$
and
$$h^{-1}\frac{\partial h}{\partial t}=\sum_{i=1}^r(\frac{\textmd{d} \lambda_i}{\textmd{d} t})e_i\otimes e^i+\sum_{i\neq j}(e^{\lambda_j-\lambda_i}-1)\alpha_{ji}e_i\otimes e^j,$$
where $\frac{\textmd{d}}{\textmd{d}t}e_i=\alpha_{ij}e_j$. Since $(\lambda_i-\lambda_j)(e^{\lambda_i-\lambda_j}-1)\geq 0$ for all $\lambda_i,\lambda_j\in \mathbb{R}$, we have
\begin{equation*}
\langle \frac{\partial}{\partial t}(\log h), h^{-1}\frac{\partial h}{\partial t}\rangle_H\geq 0.
\end{equation*}
Using the above formulas, we conclude that
\begin{equation*}\label{cur1}
\begin{split}
(\frac{\partial}{\partial t}-\widetilde{\Delta})|\Phi_{\varepsilon}|^2_H
&=-4\langle \sqrt{-1}\Lambda_{\omega } [\theta,[\theta^{*_H},\Phi_{\varepsilon}]], \Phi_{\varepsilon}\rangle_H
-2|\partial_H\Phi_{\varepsilon}|^2_H-2|\bar{\partial}_E\Phi|^2_H\\
&~~~~+2\varepsilon \langle \frac{\partial}{\partial t}(\log h),\Phi_{\varepsilon}\rangle_H\\
&\leq 0.
\end{split}
\end{equation*}
\end{proof}

We introduce the Donaldson's distance on the space of the Hermitian metrics as follows.

\begin{defn}
For any two Hermitian metrics $H$ and $K$ on the bundle $E$, we define
$$\sigma(H,K)=\mathrm{tr}(H^{-1}K)+\mathrm{tr}(K^{-1}H)-2\mathrm{rank}(E).$$
\end{defn}

It is obvious that $\sigma(H,K)\geq 0$, with equality if and only
if $H=K$. A sequence of metrics $H_{i}$ converges
to $H$ in the usual $C^{0}$ topology if and only if $\sup_{M}\sigma
(H_{i}, H)\rightarrow 0$.

\begin{prop}  \label{P2}
Let $H(t)$, $K(t)$ be two solutions of the flow \eqref{Flow}, then
$$(\widetilde{\Delta}-\frac{\partial}{\partial t})\sigma(H(t),K(t))\geq 0.$$
\end{prop}

\begin{proof}
Setting $h(t)=K(t)^{-1}H(t)$, we have
\begin{equation*}
\begin{split}
&(\widetilde{\Delta}-\frac{\partial}{\partial t})(\textmd{tr}h+\textmd{tr}h^{-1})\\
&=2\textmd{tr}(-\sqrt{-1}\Lambda_{\omega }\bar{\partial}_Ehh^{-1}\partial_Kh)+2\textmd{tr}(-\sqrt{-1}\Lambda_{\omega }\bar{\partial}_Eh^{-1}h\partial_Hh^{-1})\\
&~~+2\textmd{tr}\{h(\sqrt{-1}\Lambda_{\omega }[\theta,\theta^{*_H}-\theta^{*_K}]))+2\textmd{tr}(h^{-1}(\sqrt{-1}\Lambda_{\omega }[\theta,\theta^{*_K}-\theta^{*_H}])\}\\
&~~+2\varepsilon \textmd{tr}\{h(\log(H_0^{-1}H)-\log(H_0^{-1}K))+h^{-1}(\log(H^{-1}_0K)-\log(H^{-1}_0H))\}\\
&\geq0,
\end{split}
\end{equation*}
where we used
\begin{equation*}
\textmd{tr}\{h(\sqrt{-1}\Lambda_{\omega }[\theta,\theta^{*_H}-\theta^{*_K}])\}=|\theta h^{\frac{1}{2}}-h\theta h^{-\frac{1}{2}}|^2_{K}
\end{equation*}
and
\begin{equation*}
\textmd{tr}\{h^{-1}(\sqrt{-1}\Lambda_{\omega }[\theta,\theta^{*_K}-\theta^{*_H}])\}=|h^{-\frac{1}{2}}\theta-h^{\frac{1}{2}}\theta h^{-1}|^2_{K}.
\end{equation*}
It remains to show that
\begin{equation*}
A:=\textmd{tr}\{h(\log(H_0^{-1}H)-\log(H_0^{-1}K))+h^{-1}(\log(H^{-1}_0K)-\log(H^{-1}_0H))\}\geq 0.
\end{equation*}
Once we set $\log(H_0^{-1}H)=s_1$, $\log(H_0^{-1}K)=s_2$, we have
\begin{equation*}
\begin{split}
A&=\textmd{tr}\Big(e^{-s_2}e^{s_1}(s_1-s_2)+e^{-s_1}e^{s_2}(s_2-s_1)\Big)\\
&=\textmd{tr}\Big(e^{-s_2}(e^{s_1}-e^{s_2})(s_1-s_2)+e^{-s_1}(e^{s_2}-e^{s_1})(s_2-s_1)\Big).
\end{split}
\end{equation*}
Hence we only need to show
$$\textmd{tr}\Big(e^{-s_2}(e^{s_1}-e^{s_2})(s_1-s_2)\Big)\geq0.$$
Choose unitary basis $\{e_{\alpha}\}_{\alpha=1}^r$ such that $s_2(e_{\alpha})=\lambda_{\alpha}e_{\alpha}$. Similarly, $s_1(\widetilde{e}_{\beta})=\widetilde{\lambda}_{\beta}\widetilde{e}_{\beta}$ under the unitary basis $\{\widetilde{e}_{\beta}\}_{\beta=1}^r$. We also assume that $e_{\alpha}=b_{\alpha\beta}\widetilde{e}_{\beta}$. Direct calculation yields
\begin{equation*}
\begin{split}
\textmd{tr}\Big(e^{-s_2}(e^{s_1}-e^{s_2})(s_1-s_2)\Big)
&=\sum_{\alpha=1}^r\langle e^{-s_2}(e^{s_1}-e^{s_2})(s_1-s_2)(e_{\alpha}),e_{\alpha}\rangle_{H_0}\\
&=\sum_{\alpha=1}^r e^{-\lambda_{\alpha}} \langle \sum_{\beta=1}^rb_{\alpha\beta}(\widetilde{\lambda}_{\beta}-\lambda_{\alpha})\widetilde{e}_{\beta},
\sum_{\gamma=1}^rb_{\alpha\gamma}(e^{\widetilde{\lambda}_{\gamma}}-e^{\lambda_{\alpha}})\widetilde{e}_{\gamma}\rangle_{H_0}\\
&=\sum_{\alpha,\beta=1}^r e^{-\lambda_{\alpha}}b_{\alpha\beta}\overline{b_{\alpha\beta}}(\widetilde{\lambda}_{\beta}
-\lambda_{\alpha})(e^{\widetilde{\lambda}_{\beta}}-e^{\lambda_{\alpha}})\\
&\geq0.
\end{split}
\end{equation*}

\end{proof}

\begin{cor} \label{uniq}
Let $H$, $K$ be two Hermitian metrics satisfying \eqref{eq}, then
$$\widetilde{\Delta}\sigma(H,K)\geq 0.$$
\end{cor}

\medskip

At the end of this section, we give a proof of the identity (\ref{eq33}).  We  first recall some notation. Set $\therm(E,H_{0})=\{\eta\in \tend(E)\mid \eta^{*_{H_{0}}}=\eta\}$. Given $s \in \therm(E,H_{0})$,  we can choose  a local unitary basis  $\{e_{\alpha}\}_{\alpha=1}^{r}$ respect to $H_{0}$ and local functions $\{\lmd_{\alpha}\}_{\alpha=1}^{r}$ such that
\begin{eqnarray*}
s=\sum_{\alpha=1}^r \lmd_{\alpha} \cdot e_{\alpha}\otimes e^{\alpha}  ,
\end{eqnarray*}
 where  $\{e^{\alpha}\}_{\alpha=1}^{r}$ denotes the dual basis of $E^*$. Let   $\Psi\in C^{\infty}(\mathbb{R}\times \mathbb{R}, \mathbb{R})$ and $A=\sum\limits^r_{\alpha,\beta=1}A_{\beta}^{\alpha} e_{\alpha}\otimes e^{\beta}\in \tend(E)$.  We define:
\begin{equation*}
\Psi(\eta)(A)=\Psi(\lmd_{\alpha},\lmd_{\beta})A^{\alpha}_{\beta} e_{\alpha}\otimes e^{\beta}.
\end{equation*}

\medskip

Let $(M, g)$ be a compact Gauduchon manifold with non-empty
smooth boundary $\partial M$. Let $\varphi $ be a smooth function defined on $M$ and satisfy the boundary condition $\varphi |_{\partial M}=t$, where $t$ is a constant. By Stokes' formula, we have
\begin{equation} \label{theta01}
\begin{split}
\int_{M}|\mathrm{d}\varphi |^{2}\frac{\omega^n}{n!}
&=2\int_M(t-\varphi )\sqrt{-1}\partial \bar{\partial}\varphi \wedge\frac{\omega^{n-1}}{(n-1)!}
-2\int_M\sqrt{-1}\partial ((t-\varphi ) \bar{\partial}\varphi )\wedge\frac{\omega^{n-1}}{(n-1)!}\\
&=\int_M(t-\varphi )\tilde{\Delta }\varphi \frac{\omega^{n}}{n!}
+\int_M\sqrt{-1}  \bar{\partial}((t-\varphi)^{2} \wedge \partial \frac{\omega^{n-1}}{(n-1)!})\\
&~~~~+\int_M \sqrt{-1}\partial (\bar{\partial}(t-\varphi)^2\wedge \frac{\omega^{n-1}}{(n-1)!})-\int_M\sqrt{-1}  (t-\varphi)^{2}  \bar{\partial}(\partial \frac{\omega^{n-1}}{(n-1)!})\\
&=\int_M(t-\varphi )\tilde{\Delta }\varphi \frac{\omega^{n}}{n!}.\\
\end{split}
\end{equation}
Using (\ref{theta01}), by the same argument as that in \cite{SIM} (Lemma 5.2), we can obtain the following lemma.

\medskip

\begin{lem} [\protect {\cite[Lemma 5.2]{SIM}}] \label{SIMLEMMA}
Suppose $(X, g)$ is a non-compact Gauduchon manifold admitting  an exhaustion function $\phi$ with $\int_X|\widetilde{\Delta} \phi|\frac{\omega ^{n}}{n!}< \infty$, and suppose $\eta$ is a $(2n-1)$-form with
$\int_X|\eta|^2\frac{\omega ^{n}}{n!}< \infty$. Then if $\mathrm{d}\eta$ is integrable,
\begin{equation*}
\int_{X}\mathrm{d}\eta=0.
\end{equation*}
\end{lem}

\medskip

\begin{prop} \label{idbundle01}
Let $(E,\bar{\partial}_E,\theta)$ be a Higgs bundle with a fixed Hermitian metric $H_0$
over a Gauduchon manifold $(M, g)$. Let $H$ be a Hermitian metric on $E$ and $s:=\log(H^{-1}_0H)$. If one of the following two conditions is satisfied:

(1)Suppose that $M$ is a compact manifold with non-empty smooth boundary $\partial M$,  and $H$ is a Hermitian metric on $E$ with the same boundary condition as that of $H_{0}$, i.e. $H|_{\partial M}=H_0|_{\partial M}$.

(2)Suppose that $M$ is a non-compact manifold admitting an exhaustion function $\phi $ with $\int_M|\widetilde{\Delta} \phi|\frac{\omega^{n}}{n!}<+\infty$. Furthermore, we also assume that $|\mathrm{d}\omega^{n-1}|_{g}\in L^2(M)$, $s\in L^{\infty}(M)$ and $D_{H_{0}, \theta }^{1, 0}s \in L^{2}(M)$.

Then we have the following identity:
\begin{equation}\label{eq04021}
\int_M \mathrm{tr}(\Phi(H_0,\theta)s)\frac{\omega^n}{n!}+\int_{M}\langle \Psi(s)(\overline{\partial}_{\theta}s),\overline{\partial}_{\theta}s\rangle_{H_0}\frac{\omega^n}{n!}=\int_M \mathrm{tr}(\Phi(H,\theta)s)\frac{\omega^n}{n!},
\end{equation}
where $\overline{\partial}_{\theta}=\bar{\partial}_E+\theta$ and $\Psi $ is the function which is defined in (\ref{eq3301}).
\end{prop}

\begin{proof}
 Set $h=H^{-1}_0H =e^s$. By the definition, we have
\begin{equation}\label{def}
\textmd{tr}((\Phi(H,\theta)-\Phi(H_0,\theta))s)=\langle \sqrt{-1}\Lambda_{\omega } (\bar{\partial}(h^{-1}\partial_{H_0}h)+[\theta,\theta^{*_H}-\theta^{*_{H_0}}]),s\rangle_{H_0}.
\end{equation}
Using $\textmd{tr}(h^{-1}(\partial_{H_0}h)s)=\textmd{tr}(s\partial_{H_0}s)$, $\tr (s[\theta^{*_{ H_{0}}} , s])=0$ and $ \partial\bar{\partial}\omega^{n-1}=0$,  we have
\begin{equation} \label{theta11}
\begin{split}
&\int_{M}\langle \sqrt{-1}\Lambda_{\omega } (\bar{\partial}(h^{-1}\partial_{H_0}h)),s\rangle_{H_0}\frac{\omega^n}{n!}\\
&=\int_M\sqrt{-1}\bar{\partial}\textmd{tr}(s\partial_{H_0}s)\wedge\frac{\omega^{n-1}}{(n-1)!}
+\int_M\sqrt{-1}\textmd{tr}(h^{-1}\partial_{H_0}h\bar{\partial}s)\wedge\frac{\omega^{n-1}}{(n-1)!}\\
&=\int_M\sqrt{-1}\textmd{tr}(s\partial_{H_0}s)\wedge\overline{\partial }(\frac{\omega^{n-1}}{(n-1)!})
+\int_M \sqrt{-1}\textmd{tr}(h^{-1}\partial_{H_0}h\bar{\partial}s)\wedge \frac{\omega^{n-1}}{(n-1)!}\\
&~~~~+\int_M\sqrt{-1}\bar{\partial}(\textmd{tr}(s\partial_{H_0}s)\wedge\frac{\omega^{n-1}}{(n-1)!})\\
&=\int_M\partial (\frac{\sqrt{-1}}{2}\textmd{tr}(s^{2})\wedge\overline{\partial }(\frac{\omega^{n-1}}{(n-1)!}))+\int_M\sqrt{-1}\bar{\partial}(\textmd{tr}(s\partial_{H_0}s)\wedge\frac{\omega^{n-1}}{(n-1)!})\\
&~~~~+\int_M\sqrt{-1}\textmd{tr}(h^{-1}\partial_{H_0}h\bar{\partial}s)\wedge\frac{\omega^{n-1}}{(n-1)!}\\
&=\int_M\partial (\frac{\sqrt{-1}}{2}\textmd{tr}(s^{2})\wedge\overline{\partial }(\frac{\omega^{n-1}}{(n-1)!}))+\int_M\sqrt{-1}\bar{\partial}(\textmd{tr}(sD_{H_{0}, \theta }^{1, 0}s)\wedge\frac{\omega^{n-1}}{(n-1)!})\\
&~~~~+\int_M\sqrt{-1}\textmd{tr}(h^{-1}\partial_{H_0}h\bar{\partial}s)\wedge\frac{\omega^{n-1}}{(n-1)!}.\\
\end{split}
\end{equation}
In condition (1), by using $s|_{\partial M}=0$ and Stokes formula, in condition (2), by using Lemma \ref{SIMLEMMA}, we have
\begin{equation} \label{theta1}
\int_{M}\langle \sqrt{-1}\Lambda_{\omega } (\bar{\partial}(h^{-1}\partial_{H_0}h)),s\rangle_{H_0}\frac{\omega^n}{n!}=\int_M\sqrt{-1}\textmd{tr}(h^{-1}\partial_{H_0}h\bar{\partial}s)\wedge\frac{\omega^{n-1}}{(n-1)!}.
\end{equation}
 In \cite[p.635]{NZ}, it was proved that
\begin{equation}\label{theta21}\textmd{tr} \sqrt{-1}\Lambda_{\omega }(h^{-1}D^{1,0}_{H,\theta}h\overline{\partial}_{\theta}s)=\langle \Psi(s)(\overline{\partial}_{\theta}s),\overline{\partial}_{\theta}s\rangle_{H_0},\end{equation}
and
\begin{equation}\label{theta2}
\int_M\textmd{tr}(\sqrt{-1}\Lambda_{\omega }[\theta,\theta^{*_H}-\theta^{*_{H_0}}]s)\frac{\omega^n}{n!}
=\int_M\textmd{tr}(\sqrt{-1}h^{-1}[\theta^{*_{H_0}},h][\theta,s])\frac{\omega^{n-1}}{(n-1)!}.
\end{equation}
By (\ref{theta1}), (\ref{theta21}) and (\ref{theta2}), we obtain
\begin{equation}\label{theta22}\int_{M}\langle \sqrt{-1}\Lambda_{\omega } (\bar{\partial}(h^{-1}\partial_{H_0}h)+[\theta,\theta^{*_H}-\theta^{*_{H_0}}]),s\rangle_{H_0}\frac{\omega^n}{n!}
=\int_{M}\langle \Psi(s)(\overline{\partial}_{\theta}s),\overline{\partial}_{\theta}s\rangle_{H_0}\frac{\omega^n}{n!}.\end{equation}
Then (\ref{def}) and (\ref{theta22}) imply (\ref{eq04021}).

\end{proof}

\section{The related heat flow on  Hermitian manifolds}

In this section, we consider the existence of  long-time solutions of the related heat flow (\ref{Flow}). Let $(M, g)$ be a compact Hermitian manifold (with possibly
non-empty boundary), and $(E,\bar{\partial}_E,\theta)$ be a Higgs bundle over
$M$.
If $M$ is closed then we consider the following evolution equation:
\begin{equation} \label{Flow1}
\begin{cases}
H^{-1}\frac{\partial H}{\partial t}=-2(\sqrt{-1}\Lambda_{\omega } (F_{H}+[\theta,\theta^{*_H}])-\lambda \cdot \textmd{Id}_E+\varepsilon \log (H_0^{-1}H)),\\
H(0)=H_0.
\end{cases}
\end{equation}
If $M$ is a compact manifold with non-empty smooth boundary $\partial
M$,  for given data $\widetilde{H}$ on $\partial M$, we consider
the following Dirichlet boundary value problem:
\begin{equation} \label{Flow2}
\begin{cases}
H^{-1}\frac{\partial H}{\partial t}=-2(\sqrt{-1}\Lambda_{\omega } (F_{H}+[\theta,\theta^{*_H}])-\lambda \cdot \textmd{Id}_E+\varepsilon \log (H_0^{-1}H)),\\
H(0)=H_0,\\
H|_{\partial M}=\widetilde{H}.
\end{cases}
\end{equation}
When $\varepsilon =0$, (\ref{Flow}) is just the Hermitian-Yang-Mills flow, the existence of long-time solutions of (\ref{Flow1}) and (\ref{Flow2}) on Hermitian manifolds was proved in \cite{Z}. It is easy to see that the flow (\ref{Flow}) is strictly parabolic, so standard parabolic theory gives the short-time existence.

\begin{prop} \label{shorttime}
 For sufficiently small $T >0 $,
\eqref{Flow1} and \eqref{Flow2} have a smooth solution defined for
$0\leq t <T$.
\end{prop}

Next, following the arguments in \cite[Lemma 19]{DON85} and \cite[Lemma 6.4]{SIM},  we will prove the long-time existence.

\begin{lem}  \label{C0ofdistance}
Suppose that a smooth solution $H(t)$ of
\eqref{Flow1} or \eqref{Flow2} is defined for $0\leq t < T <+\infty $. Then
$H(t)$ converge in $C^{0}$-topology to some continuous
non-degenerate metric $H_{T}$ as $t\rightarrow T$.
\end{lem}

\begin{proof}
Given $\epsilon >0$, by continuity at $t=0$ we can
find a $\delta$ such that

$$
\sup_{M} \sigma (H(t_{0}) , H(t'_{0})) <\epsilon
$$
for $0< t_{0}, t'_{0} <\delta $. Then Proposition \ref{P2} and the maximum
principle imply that
$$
\sup_{M} \sigma (H(t) , H(t'))<\epsilon
$$
for all $t, t' > T-\delta$. This implies that $H(t)$ are
uniformly Cauchy and  converge to a continuous limiting
metric $H_{T}$. On the other hand, by Proposition \ref{P1}, we know
that
\begin{equation*}\label{z1}
\sup_{M\times [0, T)}|\sqrt{-1}\Lambda_{\omega } (F_{H(t)}+[\theta,\theta^{*_H(t)}])-\lambda \cdot \textmd{Id}_E+\varepsilon \log (H_0^{-1}H(t))|_{H(t)}<B,
\end{equation*}
where $B$ is a uniform constant depending only on the initial data $H_{0}$. Then using
\begin{equation*}
|\frac{\partial }{\partial t}(\log \textmd{tr} h)|_H\leq 2|\sqrt{-1}\Lambda_{\omega } (F_{H}+[\theta,\theta^{*_H}])-\lambda \cdot \textmd{Id}_E+\varepsilon \log (H_0^{-1}H)|_{H},
\end{equation*}
and
\begin{equation*}
|\frac{\partial }{\partial t}(\log \textmd{tr} h^{-1})|_H\leq
2|\sqrt{-1}\Lambda_{\omega } (F_{H}+[\theta,\theta^{*_H}])-\lambda \cdot \textmd{Id}_E+\varepsilon \log (H_0^{-1}H)|_{H},
\end{equation*}
one can conclude that
$\sigma (H(t), H_{0})$ are bounded uniformly on $M\times [0, T)$, therefore $H_{T}$ is a
non-degenerate metric.
\end{proof}

For further consideration, we recall the following lemma.

\begin{lem} [Lemma 3.3 in \cite{Z}]\label{C1ofh}
 Let $M$ be a compact  Hermitian manifold without boundary (with non-empty boundary). Let $H(t)$, $0\leq t
<T$, be any one-parameter family of Hermitian metrics on the Higgs bundle $E$ over $M$ (and satisfying Dirichlet boundary
condition), and suppose $H_{0}$ is the initial Hermitian metric.
If $H(t)$ converge in the $C^{0}$ topology to some continuous
metric $H_{T}$ as $t\rightarrow T$, and if $\sup_{M} |\Lambda_{\omega }
F_{H(t)}|_{H_{0}}$ is bounded uniformly in $t$, then $H(t)$ are
bounded in $C^{1}$ and also bounded in $L_{2}^{p}$ (for any
$1<p<+\infty $) uniformly in $t$.
\end{lem}

\begin{prop} \label{compactlongtime}
 \eqref{Flow1} and \eqref{Flow2}
have a unique solution $H(t)$ which exists for $0\leq t <+\infty$.
\end{prop}

\begin{proof}
 Proposition \ref{shorttime} guarantees that a solution exists for
a short time. Suppose that the solution $H(t)$ exists for $0\leq
t<T < +\infty $. By Lemma \ref{C0ofdistance},  $H(t)$ converges in $C^{0}$ to a
non-degenerate continuous limit metric $H(T)$ as $t\rightarrow T
$. Since $t<+\infty$, \eqref{dec} implies $\sup_{M} |\Lambda_{\omega }
F_{H(t)}|_{H_{0}}$ is bounded uniformly in $[0, T)$. Then by Lemma \ref{C1ofh}, $H(t)$ are
bounded in $C^{1}$ and also bounded in $L_{2}^{p}$ (for any
$1<p<+\infty $) uniformly in $t$. Since \eqref{Flow1} and \eqref{Flow2} is quadratic in the first derivative of $H$ we can
apply Hamilton's method \cite{HAMILTON} to deduce that $H(t)\rightarrow H(T)$
in $C^{\infty}$, and the solution can be continued past $T$. Then
\eqref{Flow1} and \eqref{Flow2} have a solution $H(t)$
defined for all time.

From Proposition \ref{P2} and the maximum principle, it is easy to
conclude the uniqueness of the solution.

\end{proof}

\begin{prop} \label{noncompactc1} Suppose $H(t)$ is a long-time solution of the flow (\ref{Flow}) on compact Hermitian manifold $\overline{M}$ (with nonempty smooth boundary $\partial M$). Set $h(t)=H_{0}^{-1}H(t)$ and assume that there exists a constant $\overline{C}_0$ such that
$$\sup_{(x,t)\in \overline{M}\times[0,+\infty )}|\log h|_{H_0}\leq \overline{C}_0.
$$
Then, for any compact subset $\Omega \subset \overline{M} $, there exists a uniform constant $\overline{C}_1$ depending only on $\overline{C}_0$, $d^{-1}$ and the geometry of $\tilde{\Omega }$ such that
\begin{equation}\label{CC1}\sup_{(x,t)\in \Omega \times[0,+\infty )}|h^{-1}\partial_{H_0}h|_{H_0}\leq \overline{C}_1,
\end{equation}
where $d$ is the distance of $\Omega $ to $\partial M$ and $\tilde{\Omega }=\{x \in \overline{M}|\mathrm{dist}(x, \Omega)\leq \frac{1}{2}d\}$.
\end{prop}

\begin{proof}

We will follow the argument in \protect {\cite[Lemma 2.4]{LZZ}} to get local uniform $C^{1}$-estimate.
Let $\mathcal{T}=h^{-1}\partial_{H_0}h$. Direct computations give us that
\begin{equation} \label{noncompactc1eq1}
(\widetilde{\Delta}-\frac{\partial}{\partial t}) \textmd{tr}h\geq -2\textmd{tr}(\sqrt{-1}\Lambda_{\omega }(\bar{\partial}hh^{-1}\partial_{H_0}h))
+2\textmd{tr}(h\Phi(H_0,\theta))+2\varepsilon \textmd{tr}(h\log h),
\end{equation}
\begin{equation*}
\frac{\partial}{\partial t}\mathcal{T}=\partial_H(h^{-1}\frac{\partial}{\partial t}h),
\end{equation*}
and
\begin{align} \label{noncompactc1eq2}
(\widetilde{\Delta}-\frac{\partial}{\partial t})|\mathcal{T}|^2_{H}&\geq
|\nabla_{H}\mathcal{T}|^2_{H}-\check{C}_1(|\Lambda_{\omega } F_H|_H+|F_{H_0}|_{H}+|\theta|^2_{H}+|Rm(g)|_{g}+|\nabla_g J|_g^2+\varepsilon)|\mathcal{T}|^2_H\notag\\
&~~~-\check{C}_2|\nabla_{H_0}(\Lambda_{\omega } F_{H_0})|_{H}|\mathcal{T}|_{H}-4|\nabla_{H_0}\theta|^2_{H}-\varepsilon|\log h|^2_H,
\end{align}
where $J$ is the complex structure on $M$ and positive constants $\check{C}_1$, $\check{C}_2$ depend only on the dimension $n$ and the rank $r$.
By \eqref{noncompactc1eq2} and Proposition \ref{P1}, we have
\begin{equation} \label{noncompactc1eq3}
(\widetilde{\Delta}-\frac{\partial}{\partial t})|\mathcal{T}|^2_{H}\geq
|\nabla_{H}\mathcal{T}|^2_{H}-\check{C}_3|\mathcal{T}|^2_H-\check{C}_3
\end{equation}
on the domain $\tilde{\Omega }\times [0,+\infty )$, where  $\check{C}_3$ is a uniform constant  depending only on $\overline{C}_0$, $\max_{\tilde{\Omega }}|\theta|_{H_{0}}$  and the geometry of $\tilde{\Omega }$.

Setting  $\overline{\Omega }=\{x \in \overline{M} | \mathrm{dist} (x, \Omega )\leq \frac{1}{4}d\}$. Let $\psi_1$, $\psi_2$ be non-negative cut-off functions satisfying:
 \begin{gather*}
     \psi_1
     =\begin{cases}
       0, & x\in M\backslash \overline{\Omega }, \\
       1, & x\in \Omega,
     \end{cases}
  \end{gather*}
   \begin{gather*}
    \psi_2
     =\begin{cases}
       0, & x\in M\backslash \tilde{\Omega }, \\
       1, & x\in \overline{\Omega }.
     \end{cases}
  \end{gather*}
  and
  $$|\textmd{d} \psi_i|^2+|\widetilde{\Delta} \psi_i|\leq c, \ \ i=1,2,$$
where $c=32d^{-2}$.
Consider the following test function
$$f(\cdot,t)=\psi_1^2|\mathcal{T}|^2_H+W\psi_2^2\textmd{tr}h,$$
where the constant $W$ will be chosen large enough later. It follows from \eqref{noncompactc1eq1} and \eqref{noncompactc1eq3} that
$$(\widetilde{\Delta}-\frac{\partial}{\partial t})f\geq \psi^2_2(2We^{-\overline{C}_0}-\check{C}_3-18c-8e^{2\overline{C}_0})|\mathcal{T}|^2_H-\widetilde{C}_0,$$
where $\widetilde{C}_0$ is a positive constant depending only on $\overline{C}_0$. If we choose
$$W=\frac{1}{2}e^{-\overline{C}_0}(\check{C}_3+18c+8e^{2\overline{C}_0}+1),$$
then
\begin{equation} \label{noncompactc1eq4}
(\widetilde{\Delta}-\frac{\partial}{\partial t})f\geq \psi^2_2|\mathcal{T}|^2_H-\widetilde{C}_0
\end{equation}
on $M\times [0,+\infty )$. Let $f(q,t_0)=\max_{M\times [0,+\infty)}f$. On the basis of the definition of $\psi_i$ and the uniform $C^0$-bound of $h(t)$, we may assume that:
$$(q,t_0)\in \overline{\Omega}\times (0,+\infty).$$
Of course the inequality \eqref{noncompactc1eq4} yields
\begin{equation*}
|\mathcal{T}(t_0)|^2_{H(t_0)}(q)\leq \widetilde{C}_0,
\end{equation*}
and then (\ref{CC1}).
\end{proof}

In the next part of this section, we will consider the long-time existence of the heat flow (\ref{Flow}) on some non-compact Hermitian manifold $(X, g)$. In the following, we suppose that there exists a non-negative exhaustion function $\phi$ with $\sqrt{-1}\Lambda_{\omega }\partial \bar{\partial} \phi$ bounded, i.e. $(X, g)$ satisfies the Assumption 2.
 Fix a number $\varphi$ and let $X_{\varphi}$ denote the compact space $\{x\in X |\phi(x)\leq \varphi \}$, with boundary $\partial X_{\varphi}$. Let $H_{0}$ be an initial metric on $E$ over $X$.
We consider  the following Dirichlet boundary condition
\begin{equation}\label{D1}
H|_{\partial X_{\varphi}}=H_0|_{\partial X_{\varphi}}.
\end{equation}
By Proposition \ref{compactlongtime}, on every $X_{\varphi}$, the flow \eqref{Flow} with the  above Dirichlet boundary condition and with the initial data $H_{0}$ admits a unique long-time solution $H_{\varphi}(t)$ for $0\leq t<+\infty$.

\begin{prop} \label{noncompactc0}
Suppose $H_{\varphi}(t)$ is a long-time solution of the flow (\ref{Flow}) on $X_{\varphi}$ satisfying the Dirichlet boundary condition (\ref{D1}), then
\begin{equation}\label{c0key}|\log h|_{H_0}(x,t)\leq \frac{1}{\varepsilon}\max_{X_{\varphi}}|\Phi(H_0,\theta)|_{H_0},\ \ \forall
(x,t)\in X_{\varphi}\times[0,+\infty ).\end{equation}
where $h(t)=H_{0}^{-1}H_{\varphi}(t)$, $\overline{C}_0$ is a uniform constant depending only on $\varepsilon^{-1}$ and the initial data $\max_{X_{\varphi}}|\Phi(H_0,\theta)|_{H_0}$.
\end{prop}

\begin{proof}

By a direct calculation, we have
\begin{align}
\langle H^{-1}\frac{\partial H}{\partial t},\log h\rangle_{H_0}&=-2\langle \sqrt{-1}\Lambda_{\omega } (F_{H}+[\theta,\theta^{*_H}])-\lambda\cdot \textmd{Id}_E+\varepsilon \log h,\log h\rangle_{H_0}\notag\\
&=-2\langle \Phi(H_0,\theta)+\sqrt{-1}\Lambda_{\omega }(\bar{\partial}(h^{-1}\partial_{H_0}h)+[\theta,\theta^{*_H}-\theta^{*_{H_0}}])+\varepsilon \log h,\log h\rangle_{H_0}\notag\\
&\leq -2 \langle \Phi(H_0,\theta)+\sqrt{-1}\Lambda_{\omega }\bar{\partial}(h^{-1}\partial_{H_0}h)+\varepsilon \log h,\log h\rangle_{H_0},\notag
\end{align}
where we have used the inequality ((2.6) in \cite{NZ})
\begin{equation*}
\langle \sqrt{-1}\Lambda_{\omega}[\theta,\theta^{*_H}-\theta^{*_{H_0}}] , \log h\rangle_{H_0}\geq 0.
\end{equation*}
On the other hand, it is easy to check that
\begin{equation*}
\langle H^{-1}\frac{\partial H}{\partial t},\log h\rangle_{H_0}=\langle h^{-1}\frac{\partial h}{\partial t},\log h\rangle_{H_0}=\frac{1}{2}\frac{\partial}{\partial t}|\log h|^2_{H_0}
\end{equation*}
and
\begin{equation*}
\langle \sqrt{-1}\Lambda_{\omega }\bar{\partial}(h^{-1}\partial_{H_0}h),\log h\rangle_{H_0}\geq -\frac{1}{2}\widetilde{\Delta}(|\log h|^2_{H_0}).
\end{equation*}
Then
\begin{equation*}
\frac{1}{4}(\frac{\partial}{\partial t}-\widetilde{\Delta})(|\log h|^2_{H_0})\leq -\varepsilon |\log h|^2_{H_0}+|\Phi(H_0,\theta)|_{H_0}|\log h|_{H_0},
\end{equation*}
which together with the maximum principle implies (\ref{c0key}).

\end{proof}

\begin{lem} [\protect {\cite[Lemma 6.7]{SIM}}] \label{SIMLEM}
Suppose $u$ is a function on some $X_{\varphi}\times[0,T]$, satisfying
$$(\widetilde{\Delta}-\frac{\partial}{\partial t})u\geq 0,\ \ \ u|_{t=0}=0,$$
and suppose there is a bound $\sup_{X_{\varphi}}u\leq C_1$. Then we have
$$u(x,t)\leq \frac{C_1}{\varphi}(\phi(x)+C_2t),$$
where $C_2$ is the bound of $\widetilde{\Delta}\phi$ in Assumption 2.

\end{lem}

\medskip

In the following, we assume that there exists a constant $C$ such that $\sup_{X}|\Phi(H_0,\theta)|_{H_0}\leq C$. For any compact subset $\Omega \subset X$, there exists a constant $\varphi_{0}$ such that $\Omega \subset X_{\varphi_{0}}$.
Let $H_{\varphi_{1}}(t)$ and $H_{\varphi_2}(t)$ be the long-time solutions of the flow (\ref{Flow}) satisfying the Dirichlet boundary condition (\ref{D1}) for $\varphi_{0}<\varphi_{1}<\varphi_{2}$. Let $u=\sigma(H_{\varphi_{1}},H_{\varphi_2})$. Proposition \ref{noncompactc0} gives a uniform bound on $u$, and $u$ is a subsolution for the heat operator with $u(0)=0$. By Lemma \ref{SIMLEM}, we have
\begin{equation}\sigma(H_{\varphi},H_{\varphi_1})\leq C_{1}\frac{(\varphi_{0}+C_{2}T)}{\varphi}\end{equation}
on $X_{\varphi_{0}}\times [0, T]$. Then $H_{\varphi}$ is a Cauchy sequence on $X_{\varphi_{0}}\times [0, T]$ for $\varphi\rightarrow \infty$.
Proposition \ref{noncompactc0} and Proposition \ref{noncompactc1} give the uniform $C^0$ and local $C^1$ estimates of $H_{\varphi}(t)$. One can get the local uniform $C^{\infty}$-estimate of $H_{\varphi }(t)$ by the standard Schauder estimate of the parabolic equation. It should be point out that by applying the parabolic Schauder estimate, one can only get the uniform $C^{\infty}$-estimate of $h(t)$ on $X_{\varphi}\times[\tau,T]$, where $\tau>0$ and the uniform estimate depends on $\tau^{-1}$. To fix this, one can use the maximum principle to get  a local uniform bound on the curvature $|F_H|_H$, then apply the elliptic estimates to get local uniform $C^{\infty}$-estimates. We will omit this step here, since it is similar to \protect{\cite[Lemma 2.5]{LZZ}}.
By choosing a subsequence $\varphi \rightarrow \infty $, we have that $H_{\varphi}(t)$  converge in $C_{loc}^{\infty}$-topology to a long-time solution $H(t)$ of the heat flow (\ref{Flow}) on $X$. So, we obtain the following theorem.

\begin{thm} \label{noncompactthm}
Let $(E,\bar{\partial}_E,\theta)$ be a Higgs bundle with fixed Hermitian metric $H_0$
over a Hermitian manifold $(X,g)$ satisfying the Assumptions 2.  Suppose $\sup\limits_{X}|\Phi(H_0,\theta)|_{H_0}<+\infty$, then, on the whole $X$, the flow \eqref{Flow} has a long-time solution $H(t)$ satisfying:
\begin{equation}\label{c0key2}\sup_{(x,t)\in X\times[0,+\infty )}|\log h|_{H_0}(x,t)\leq \frac{1}{\varepsilon}\sup_{X}|\Phi(H_0,\theta)|_{H_0}.\end{equation}
\end{thm}

\section{Poisson equations on the non-compact manifold}

In this section, we are devoted to solve the equation $\widetilde{\Delta} f=\psi$ on a class of non-compact Gauduchon manifold.
Since the difference of the complex Laplacian and the Beltrami-Laplcaian  is given by a linear first order differential operator, the following proposition should be well known, it also can be proved in the same way as that in Theorem \ref{comthm}.

\begin{prop} \label{aprop1}
Let $(M,g)$ be a compact Hermitian manifold with non-empty boundary $\partial M$. Suppose that $\psi \in C^{\infty}(M)$, then for any function $\widetilde{f}$ on the restriction to $\partial M$, there is a unique function $f\in C^{\infty}(M)$ which satisfies the equation $\widetilde{\Delta} f=\psi+\varepsilon f$  and $f=\widetilde{f}$ on $\partial M$ for any $\varepsilon > 0$.
\end{prop}

Let $(X,g)$ be a non-compact Gauduchon manifold with finite volume and a non-negative exhaustion function $\phi$. By Proposition \ref{aprop1}, we know that the following Dirichlet problem is solvable on $X_{\varphi}$, i.e.
\begin{equation*}
\begin{cases}
\widetilde{\Delta}f_{\varphi}-\varepsilon f_{\varphi}-\psi=0, \ \  \forall x\in X_{\varphi},\\
f_{\varphi}(x)|_{\partial X_{\varphi}}=0.
\end{cases}
\end{equation*}
By simple calculations, we have
\begin{equation*}
\widetilde{\Delta}|f_{\varphi}|^2\geq 2|f_{\varphi}|(\varepsilon|f_{\varphi}|-|\psi|).
\end{equation*}
The maximum principle implies:
\begin{equation*}\label{poisson01}
\max_{X_{\varphi}}|f_{\varphi}|\leq \frac{1}{\varepsilon}\sup_{X_{\varphi}}|\psi|.
\end{equation*}
By (\ref{theta01}), we have
\begin{equation*} \label{theta010}
\begin{split}
\int_{X_{\varphi }}|\mathrm{d} f_{\varphi} |^{2}\frac{\omega^n}{n!}
&=-\int_{X_{\varphi }} f_{\varphi}\tilde{\Delta }f_{\varphi} \frac{\omega^{n}}{n!}\\
& \leq \frac{1}{\varepsilon}\sup_{X_{\varphi}}|\psi|^{2}\mathrm{Vol}(X_{\varphi }, g).\\
\end{split}
\end{equation*}
Then, by using  the standard elliptic estimates, we can prove that, by choosing a subsequence,  $f_{\varphi }$ converge in $C_{loc}^{\infty}$-topology to a solution on whole $X$, i.e. we  prove the following proposition.

\begin{prop} \label{aprop2}
Let $(X,g)$ be a non-compact Gauduchon manifold with finite volume and a non-negative exhaustion function $\phi$. Suppose that $\psi \in C^{\infty}(X)$ satisfies $\sup\limits_X|\psi|<+\infty$.  For any  $\varepsilon>0$, there is a  function $f\in C^{\infty}(X)$ which satisfies the equation
\begin{equation}\label{040301}
\widetilde{\Delta} f=\psi+\varepsilon f
\end{equation}
with
\begin{equation}\label{040302}
\sup_X|f|\leq \frac{1}{\varepsilon }\sup_X|\psi|
\end{equation}
and
\begin{equation} \label{040303}
\int_{X}|\mathrm{d}f|^2\frac{\omega^n}{n!}\leq \frac{1}{\varepsilon}(\sup_X|\psi|)^2\mathrm{Vol}(X,g).
\end{equation}
\end{prop}

\medskip

Now we are ready to solve the Poisson equation on the non-compact Gauduchon manifold.

\medskip

\begin{prop} \label{poisson}
Let $(X, g)$ be a non-compact Gauduchon manifold satisfying Assumptions 1,2,3 and $|\mathrm{d}\omega^{n-1}|_{g}\in L^2(X)$. Suppose that $\psi \in C^{\infty}(X)$ satisfies $\int_X \psi=0$ and $\sup\limits_X|\psi|<+\infty$. Then there is a function $f\in C^{\infty}(X)$ which satisfies the Possion equation \begin{equation}\label{poisson1}\widetilde{\Delta} f=\psi , \end{equation}
\begin{equation}\label{int01}
\int_{X}|\mathrm{d}f|^{2}\frac{\omega^{n}}{n!}<+\infty
\end{equation}
and $\sup_X|f|<+\infty$.
\end{prop}

\begin{proof}
By a direct calculation, we have
\begin{equation*}
\widetilde{\Delta}\log(e^f+e^{-f})\geq -|\widetilde{\Delta}f|.
\end{equation*}
On the other hand, it is easy to check that
\begin{equation*}
|f|\leq \log(e^f+e^{-f})\leq |f|+\log 2.
\end{equation*}
From Proposition \ref{aprop2}, for any  $\varepsilon>0$, we have a solution $f_{\varepsilon}$ of the equation (\ref{040301})  and $f_{\varepsilon}$ satisfies (\ref{040302}).
By Assumption 3, we have
\begin{equation*}
\sup_X|f_{\varepsilon}|\leq \sup_X\log(e^{f_{\varepsilon}}+e^{-f_{_{\varepsilon}}})\leq \widetilde{C}_1\int_X |f_{\varepsilon}|+\widetilde{C}_2,
\end{equation*}
where constants $\widetilde{C}_1$ and $\widetilde{C}_2$ depend only on $\sup_X|\psi|$ and $\textmd{Vol}(X)$.

In the following, we will use a contradiction argument to prove that
$\| f_{\varepsilon}\|_{C^0}$ is uniform bounded.
If $\| f_{\varepsilon}\|_{C^0}$ is unbounded, then there exists a subsequence $\varepsilon\rightarrow 0$, such that $\| f_{\varepsilon}\|_{L^2}\rightarrow +\infty$. Set $u_{\varepsilon}=f_{\varepsilon}/\parallel f_{\varepsilon}\parallel_{L^2}$. It follows that
$$\| u_{\varepsilon}\|_{L^2}=1 \ \ \textmd{and}\ \ \sup_X|u_{\varepsilon}|<\widetilde{C}_3 < +\infty ,$$
where $\widetilde{C}_3$ is a uniform constant depending only on $\sup_X|\psi|$ and $\textmd{Vol}(X)$.
Using the conditions $\partial \bar{\partial}\omega^{n-1}=0$, $|\mathrm{d}\omega^{n-1}|_{g}\in L^2(X)$, \eqref{040302}, \eqref{040303}, and Lemma \ref{SIMLEMMA}, one can check that
\begin{equation}\label{aeq5}
\int_X f_{\varepsilon}\widetilde{\Delta} f_{\varepsilon} \frac{\omega^n}{n!}=-\int_X|\mathrm{d} f_{\varepsilon}|^2\frac{\omega^n}{n!}.
\end{equation}
Substituting the perturbed equation into \eqref{aeq5}, we have
\begin{equation*}\label{int02}
\int_X|\mathrm{d} u_{\varepsilon}|^2 \frac{\omega^{n}}{n!}=-\varepsilon-\frac{1}{\| f_{\varepsilon}\|_{L^2}}\int_X u_{\varepsilon}\psi \frac{\omega^{n}}{n!}.
\end{equation*}
Then, by passing to a subsequence, we have that $u_{\varepsilon }$ converges weakly to $u_{\infty}$ in $L^2_1$ as $\varepsilon \rightarrow 0$, and $u_{\infty}$ is constant almost everywhere. Note that for any relatively compact $Z\subset X$, $L^2_1\rightarrow L^2(Z)$ is compact. So
$$\int_Z |u_{\varepsilon}|^2\rightarrow \int_Z|u_{\infty}|^2.$$
Recalling $\sup_X|u_{\varepsilon_{i}}|<\widetilde{C}_3< +\infty$ and $X$ has finite volume, so for a small $\epsilon>0$, we have
\begin{equation*}
\int_{X\backslash Z}|u_{\varepsilon}|^2<\epsilon,
\end{equation*}
when $Z$ is big enough. Thus $1\geq \int_Z|u_{\infty}|^2\geq 1-\epsilon$. So, we have
$$u_{\infty}=\textmd{const}.\neq 0 \ \ a.e..$$
Using the conditions $\partial \bar{\partial}\omega^{n-1}=0$, $|\mathrm{d}\omega^{n-1}|_{g}\in L^2(X)$, \eqref{040302}, \eqref{040303} and Lemma \ref{SIMLEMMA}, it is easy to check that
\begin{equation*}
\int_X \widetilde{\Delta} f_{\varepsilon} \frac{\omega^{n}}{n!}=0.
\end{equation*}
Then combining $\widetilde{\Delta} f_{\varepsilon}+\varepsilon f_{\varepsilon}+\psi=0$ and $\int_X\psi=0$, we have
\begin{equation*}
\int_X f_{\varepsilon} \frac{\omega^{n}}{n!}=0,
\end{equation*}
and
\begin{equation*}
\int_X u_{\varepsilon}\frac{\omega^{n}}{n!}=0.
\end{equation*}
Then, we can obtain
\begin{equation*}
\int_X u_{\infty}\frac{\omega^{n}}{n!}=0.
\end{equation*}
We get a contradiction, so we have proved that $\| f_{\varphi }\|_{C^0}$ is bounded uniformly when $\varepsilon$ goes to zero. By standard elliptic estimates, we obtain, by choosing a subsequence $f_{\varepsilon}$ must converge to a smooth function $f_{\infty}$ in $C_{loc}^{\infty}$-topology as $\varepsilon \rightarrow 0$, and $f_{\infty}$ satisfies the equation (\ref{poisson1}). (\ref{aeq5}) implies (\ref{int01}). This completes the proof of Proposition \ref{poisson}.
\end{proof}

\section{ Solvability of the perturbed equation}

We first solve the Dirichlet problem for the perturbed equation, i.e. we obtain the following theorem.

\begin{thm} \label{comthm}
Let $(E,\bar{\partial}_E,\theta)$ be a Higgs bundle with fixed Hermitian metric $H_0$
over the compact Gauduchon manifold $\overline{M}$ with non-empty
boundary $\partial M$. There is a unique
Hermitian metric $H$ on $E$ such that
\begin{equation}\label{peq01}
\sqrt{-1}\Lambda_{\omega } (F_{H}+[\theta,\theta^{*_H}])-\lambda \cdot \mathrm{Id}_E+\varepsilon \log (H_0^{-1}H)=0, \ \ \ \ H|_{\partial M}=H_{0},
\end{equation}
for any $\varepsilon \geq 0$.
When $\varepsilon >0$,  we have
\begin{equation}\label{DC0}
\max_{x\in \overline{M}}|s|_{H_0}(x)\leq \frac{1}{\varepsilon}\max_{\overline{M}}|\Phi(H_0,\theta)|_{H_0}.
\end{equation}
and
\begin{equation} \label{eq0403}
\|D_{H_{0}, \theta }^{1, 0}s\|_{L^2(\overline{M})}=\|\overline{\partial}_{\theta}s\|_{L^2(\overline{M})}\leq C(\varepsilon^{-1},\Phi(H_0,\theta),\mathrm{Vol}(M)),
\end{equation}
where $s=\log (H_{0}^{-1}H)$. Furthermore, if the initial metric $H_{0}$ satisfies the following condition
\begin{equation}\label{trace3}
\mathrm{tr}(\sqrt{-1}\Lambda_{\omega } (F_{H}+[\theta,\theta^{*_H}])-\lambda \cdot \mathrm{Id}_E)=0,
\end{equation}
then $\mathrm{tr}(s)=0$ and $H$ also satisfies the condition (\ref{trace3}).
\end{thm}

\begin{proof}
Proposition \ref{compactlongtime} guaranteed the existence of long-time solution $H(t)$ of the heat equation (\ref{Flow2}).
By Proposition \ref{P1}, we have
\begin{equation} \label{thmeq1}
(\widetilde{\triangle }-\frac{\partial}{\partial
t})|\sqrt{-1}\Lambda_{\omega } (F_{H(t)}+[\theta,\theta^{*_{H(t)}}])-\lambda \cdot \textmd{Id}_E+\varepsilon \log (H_0^{-1}H(t))|_{H(t)}\geq 0.
\end{equation}
If the initial metric $H_{0}$ satisfies  the condition (\ref{trace3}), by (\ref{trace}) and the maximum principle, we know that $H(t)$ must satisfy
\begin{equation*}
\tr \{\sqrt{-1}\Lambda_{\omega }(F_{H(t)}+[\theta,\theta^{\ast_{H(t)}}])-\lambda \cdot \textmd{Id}_E+\varepsilon \log (H_{0}^{-1}H(t))\}=0.
\end{equation*}
Then, we have
\begin{equation*}
\tr (\log H_{0}^{-1} H(t))=0
\end{equation*}
and $H(t)$ satisfies the condition (\ref{trace3}) for all $t\geq 0$.

By \protect {\cite[Chapter 5, Proposition 1.8]{T}}, one can solve the following Dirichlet problem on $M$:
\begin{equation} \label{thmeq2}
\widetilde{\triangle }v =-
|\sqrt{-1}\Lambda_{\omega } (F_{H_0}+[\theta,\theta^{*_{H_0}}])-\lambda \cdot \textmd{Id}_E|_{H_{0}}, \ \
v|_{\partial M}=0.
\end{equation}
Set $w(x, t)=\int_{0}^{t}|\sqrt{-1}\Lambda_{\omega } (F_{H}+[\theta,\theta^{*_H}])-\lambda \cdot \textmd{Id}_E+\varepsilon \log (H_0^{-1}H)|_{H}(x, \rho) \mathrm{d}\rho -v(x)$. From \eqref{thmeq1}, \eqref{thmeq2}, and the boundary
condition satisfied by $H$ implies that, for $t>0$,
$|\sqrt{-1}\Lambda_{\omega } (F_{H}+[\theta,\theta^{*_H}])-\lambda \cdot \textmd{Id}_E+\varepsilon \log (H_0^{-1}H)|_{H}(x, t)$ vanishes on the
boundary of $M$, it is easy to check that $w(x, t)$ satisfies
$$
(\widetilde{\triangle } -\frac{\partial
}{\partial t })w(x, t)\geq 0, \ \ w(x, 0)=-v(x), \ \ w(x, t)|_{\partial M }=0.
$$
By the maximum principle, we have
\begin{equation} \label{thmeq3}
\int_{0}^{t}|\sqrt{-1}\Lambda_{\omega } (F_{H}+[\theta,\theta^{*_H}])-\lambda \cdot \textmd{Id}_E+\varepsilon \log (H_0^{-1}H)|_{H}(x, \rho) \mathrm{d}\rho \leq
\sup_{y\in M} v(y),
\end{equation}
for any $x\in M$, and $0<t <+\infty $.

Let $t_{1}\leq t \leq t_{2}$, and let $\bar{h}(x, t)= H^{-1}(x,
t_{1})H(x, t)$. It is easy to check that
\begin{equation*}
\frac{\partial }{\partial t} \log \textmd{tr}(\bar{h})\leq
2|\sqrt{-1}\Lambda_{\omega } (F_{H}+[\theta,\theta^{*_H}])-\lambda \cdot \textmd{Id}_E+\varepsilon \log (H_0^{-1}H)|_{H}.
\end{equation*}
By integration, we have
\begin{equation*}
\begin{split}
&\textmd{tr} (H^{-1}(x, t_{1})H(x, t))\\
&~~\leq r\exp{(2\int_{t_{1}}^{t}|\sqrt{-1}\Lambda_{\omega } (F_{H}+[\theta,\theta^{*_H}])-\lambda \cdot \textmd{Id}_E+\varepsilon \log (H_0^{-1}H)|_{H}
\mathrm{d}\rho)}.
\end{split}
\end{equation*}
We have a similar estimate for $\textmd{tr} (H^{-1}(x, t)H(x, t_{1}))$.
Combining them we have
\begin{equation} \label{thmeq4}
\begin{split}
&\sigma (H(x , t), H(x , t_{1}) )\\
&~~\leq 2r
(\exp{(2\int_{t_{1}}^{t}|\sqrt{-1}\Lambda_{\omega }(F_{H}+[\theta,\theta^{*_H}])-\lambda \cdot \textmd{Id}_E+\varepsilon \log (H_0^{-1}H)|_{H}
\mathrm{d}\rho)} -1 ).
\end{split}
\end{equation}
 By \eqref{thmeq3} and \eqref{thmeq4},  we have that $H(t)$ converge in the $C^{0}$ topology
to some continuous metric $H_{\infty}$ as $t\longrightarrow
+\infty$. From  Lemma \ref{C1ofh}, we know that $H(t)$ are bounded uniformly
in $C^{1}_{loc}$ and also bounded uniformly in $L_{2, loc}^{p}$ (for any $1<p<+\infty $)
. On the other hand, we have known that $|H^{-1}\frac{\partial H
}{\partial t}|$ is bounded uniformly. Then, the standard elliptic
regularity implies that there exists a subsequence
$H(t)\longrightarrow H_{\infty}$ in $C^{\infty}_{loc}$-topology. From
formula \eqref{thmeq3}, we know that $H_{\infty}$ is the desired
Hermitian metric satisfying the boundary condition. By
Corollary \ref{uniq} and the maximum principle, it is easy to conclude
the uniqueness of solution.

If $\varepsilon >0$, (\ref{c0key}) in Proposition \ref{noncompactc0} implies (\ref{DC0}). By the definition, it is easy to check
\begin{equation*}
|\overline{\partial}_{\theta}s|_{H_{0}}^{2}\leq \tilde{C}\langle \Psi(s)(\overline{\partial}_{\theta}s),\overline{\partial}_{\theta}s\rangle_{H_0},
\end{equation*}
where $\tilde{C}$ is a positive constant depending only on the $L^{\infty}$-bound of $s$.
By the identity (\ref{eq04021}) in Proposition \ref{idbundle01} and the equation (\ref{peq01}), we have
\begin{equation}\label{eq04021000}
\begin{split}
\int_{M}|\overline{\partial}_{\theta}s|_{H_{0}}^{2}\frac{\omega^n}{n!}&\leq \tilde{C}
\int_{M}\langle \Psi(s)(\overline{\partial}_{\theta}s),\overline{\partial}_{\theta}s\rangle_{H_0}\frac{\omega^n}{n!}\\
&=\tilde{C}\int_M (-\mathrm{tr}(\Phi(H_0,\theta)s)-\varepsilon |s|_{H_{0}}^{2})\frac{\omega^n}{n!}\\
& \leq \tilde{C}\frac{1}{\varepsilon }\cdot \sup_{M}|\Phi(H_0,\theta)|_{H_0}^{2}\cdot \mathrm{Vol}(M, g).\\
\end{split}
\end{equation}
Then (\ref{eq04021000}) implies (\ref{eq0403}).
\end{proof}

Let $X$ be a non-compact Gauduchon manifold, $\{ X_{\varphi} \}$ an exhausting sequence of
compact sub-domains of $X$. Suppose $(E,\bar{\partial}_E,\theta)$ is a Higgs bundle over $X$ and $H_{0}$ is a Hermitian metric on $E$. By Theorem \ref{comthm}, we know that the following Dirichlet problem is solvable on $X_{\varphi}$, i.e. there exists a Hermitian metric $H_{\varphi}(x)$ such that
\begin{equation*}\begin{cases}
\sqrt{-1}\Lambda_{\omega } (F_{H_{\varphi}}+[\theta,\theta^{*_{H_{\varphi}}}])-\lambda \cdot \textmd{Id}_E+\varepsilon \log (H_0^{-1}H_{\varphi})=0, \forall x\in X_{\varphi},\\
H_{\varphi}(x)|_{\partial X_{\varphi}}=H_0(x).
\end{cases}
\end{equation*}
In order to prove that we can pass to limit and eventually obtain a solution on the whole manifold $X$, we need some a priori estimates. The key is the $C^0$-estimate.

We denote $h_{\varphi}=H_0^{-1}H_{\varphi}$. Theorem \ref{comthm} implies:
\begin{equation*}
\sup_{x\in X_{\varphi}}|\log h_{\varphi}|_{H_0}(x)\leq \frac{1}{\varepsilon}\max_{X_{\varphi}}|\Phi(H_0,\theta)|_{H_0},
\end{equation*}
For any compact subset $\Omega\subset X$, we can choose a $\varphi_{0}$ such that $\Omega\subset X_{\varphi_{0}}$.
By Proposition \ref{noncompactc1}, we have the following local uniform $C^{1}$-estimates, i.e. for any $\varphi >\varphi_{0}$, there exists
\begin{equation*}\label{CC10}
\sup_{x\in \Omega }|h_{\varphi}^{-1}\partial_{H_0}h_{\varphi}|_{H_0}\leq \hat{C}_1,
\end{equation*}
where $\hat{C}_{1}$ is a uniform constant independent on $\varphi$. The perturbed equation (\ref{eq}) and standard elliptic theory give us uniform local higher order estimates. Then, by passing to a subsequence, $H_{\varphi}$ converge in $C_{loc}^{\infty}$ topology  to a metric $H_{\infty}$ which is a solution of the perturbed equation \eqref{eq} on the whole manifold $X$. Therefore we complete the proof of the following theorem.

\begin{thm} \label{sol}
Let $(E,\bar{\partial}_E,\theta)$ be a Higgs bundle with fixed Hermitian metric $H_0$
over the non-compact Gauduchon manifold $(X,g)$ with finite volume. Suppose there exists a non-negative exhaustion function $\phi$ on $X$ and $\sup\limits_X|\Phi(H_0,\theta)|_{H_0}<+\infty$, then for any $\varepsilon>0$, there exists a metric $H$ such that
$$\sqrt{-1}\Lambda_{\omega } (F_{H}+[\theta,\theta^{*_H}])-\lambda \cdot \mathrm{Id}_E+\varepsilon \log (H_0^{-1}H)=0,$$
\begin{equation} \label{eq0404}
\sup_{x\in X}|\log H_0^{-1}H|_{H_0}(x)\leq \frac{1}{\varepsilon}\sup_{X}|\Phi(H_0,\theta)|_{H_0},
\end{equation}
and
\begin{equation} \label{eq0405}
\|\overline{\partial}_{\theta}(\log H^{-1}_0H)\|_{L^2}\leq C(\varepsilon^{-1},\Phi(H_0,\theta),\mathrm{Vol}(X)).
\end{equation}
Furthermore, if the initial metric $H_{0}$ satisfies the  condition (\ref{trace3})
then $\mathrm{tr} \log (H^{-1}_0H)=0$ and $H$ also satisfies the condition (\ref{trace3}).
\end{thm}

\section{Proof of the theorems}

Let $(X,g)$ be a non-compact Gauduchon manifold satisfying the Assumptions 1,2,3, and $|\mathrm{d}\omega^{n-1}|_{g}\in L^2(X)$,  $(E,\bar{\partial}_E,\theta)$ be a Higgs bundle over $X$. Fixing a proper background Hermitian metric $K$ satisfying $\sup\limits_X|\Lambda_{\omega } F_{K,\theta}|_{K}<+\infty$ on $E$. By Proposition \ref{poisson}, we can solve the following Poisson equation on $(X, g)$:
\begin{equation*} \label{lastsection4}
\sqrt{-1}\Lambda_{\omega } \bar{\partial}\partial f=
-\frac{1}{r}\textmd{tr}(\sqrt{-1}\Lambda_{\omega } F_{K,\theta}-\lambda_{K,\omega}\cdot \textmd{Id}_E),
\end{equation*}
where
$$
\lambda_{K,\omega}=\frac{\sqrt{-1}\int_X\textmd{tr}(\Lambda_{\omega } F_{K,\theta})
\frac{\omega^{n}}{n!}}{\textmd{rank}(E)\textmd{Vol}(X)}.
$$

By  conformal change $\overline{K}=e^{f}K$, we can check  that $\overline{K}$ satisfies
\begin{equation} \label{lastsection3}
\textmd{tr}( \sqrt{-1}\Lambda_{\omega } (F_{\overline{K}}+[\theta,\theta^{*_{\overline{K}}}])-\lambda_{K,\omega} \cdot \textmd{Id}_E)=0.
\end{equation}
By the definition and properties of $f$, it is easy to check that if $(E,\bar{\partial}_E,\theta)$ is $K$-analytic stable then it must be $\overline{K}$-analytic stable. So, in the following we can assume that the initial metric $K$ satisfies the condition (\ref{lastsection3}).

From Theorem \ref{sol}, we can solve the following perturbed equation
\begin{equation} \label{lastsection1}
L_{\varepsilon}(h_{\varepsilon }):= \sqrt{-1}\Lambda_{\omega } (F_{H_{\varepsilon }}+[\theta,\theta^{*_{H_{\varepsilon }}}])
-\lambda_{K,\omega} \cdot \textmd{Id}_E+\varepsilon \log h_{\varepsilon }=0,
\end{equation}
where $h_{\varepsilon }=K^{-1}H_{\varepsilon }=e^{s_{\varepsilon }}$. Since the initial metric $K$ satisfies the condition (\ref{lastsection3}), then we have
\begin{equation*}\label{trace4}
\log \det (h_{\varepsilon })=\mathrm{tr}(s_{\varepsilon})=0
\end{equation*}
 and
\begin{equation*}
\textmd{tr}( \sqrt{-1}\Lambda_{\omega } (F_{H_{\varepsilon }}+[\theta,\theta^{*_{H_{\varepsilon }}}])
-\lambda_{K,\omega} \cdot \textmd{Id}_E)=0.
\end{equation*}

 \medskip

\begin{lem} \label{lastsectionlemma}
\begin{equation}\label{L1}\sup_{X}|\log h_{\varepsilon}|\leq  C_7\| \log h_{\varepsilon}\|_{L^2(X)}+C_8,\end{equation}
where $C_7$ and $C_8$ are positive constants independent on $\varepsilon$.
\end{lem}

\begin{proof}
By \protect{\cite[Lemma 3.1 (d)]{SIM}}, we have
\begin{equation*}
\widetilde{\Delta}\log (\textmd{tr} h_{\varepsilon} +\tr h^{-1}_{\varepsilon})\geq -2(|\Lambda_{\omega } F_{H_{\varepsilon},\theta}|_{H_{\varepsilon}}+|\Lambda_{\omega } F_{K,\theta}|_K).
\end{equation*}
From (\ref{eq0404}) and (\ref{lastsection1}), it is easy to check that $|\Lambda_{\omega } F_{H_{\varepsilon},\theta}|_{H_{\varepsilon }}$ is uniformly bounded. On the other hand, we have
\begin{equation*}
\log (\frac{1}{2r}(\mathrm{tr} h_{\varepsilon} + \mathrm{tr} h^{-1}_{\varepsilon}))\leq |\log h_{\varepsilon }|\leq r^{\frac{1}{2}}\log (\mathrm{tr} h_{\varepsilon} + \mathrm{tr} h^{-1}_{\varepsilon}),
\end{equation*}
 Then by Assumption 3, we have (\ref{L1}).
\end{proof}

{\bf Proof of Theorem \ref{theorem1}}

When $(E,\bar{\partial}_E,\theta)$ is $K$-stable, we will show that, by choosing a subsequence, $H_{\varepsilon}$ converge to a Hermitian-Einstein metric $H$ in $C_{loc }^{\infty}$ as $\varepsilon \rightarrow 0$. By the local $C^{1}$-estimates in Proposition \ref{noncompactc1}, the standard elliptic estimates and the identity (\ref{eq04021}) in Proposition \ref{idbundle01}, we only need to obtain a uniform $C^{0}$-estimate.
By Lemma \ref{lastsectionlemma}, the key is to get a uniform $L^{2}$-estimate for $\log h_{\varepsilon}$, i.e. there exists a constant $\hat{C}$ independent of $\varepsilon$, such that
\begin{equation}\label{L102}
\|\log h_{\varepsilon}\|_{L^2}= \int_{X}|\log h_{\varepsilon}|_{H_{\varepsilon}}\frac{\omega^{n}}{n!}\leq \hat{C}
\end{equation}
for all $0<\varepsilon \leq 1$. We prove (\ref{L102}) by contradiction. If not, there would exist a subsequence $\varepsilon_{i}\rightarrow 0$ such that
\begin{equation*}
\|\log h_{\varepsilon_i}\|_{L^2}\rightarrow +\infty.
\end{equation*}
Once we set
$$s_{\varepsilon_i}=\log h_{\varepsilon_i}, \ \ l_i=\|s_{\varepsilon_i}\|_{L^2}, \ \ u_{\varepsilon_i}=\frac{s_{\varepsilon_i}}{l_i},$$
we have
\begin{equation*}
\textmd{tr} (u_{\varepsilon_i})=0, \ \| u_{\varepsilon_i}\|_{L^2}=1.
\end{equation*}
Then combining Lemma \ref{lastsectionlemma}, we also have
\begin{equation} \label{uc0}
\sup\limits_X|u_{\varepsilon_i}|\leq \frac{1}{l_i}(C_7l_i+C_8)<C_{10}<+\infty.
\end{equation}

$\bullet$ $Step \ 1$ \ \ We show that $\| u_{\varepsilon_i}\|_{L^2_1}$ are uniformly bounded. Since $\| u_{\varepsilon_i}\|_{L^2}=1$, we only need to prove $\| \overline{\partial}_{\theta}u_{\varepsilon_i}\|_{L^2}$ are uniformly bounded.

From Theorem \ref{sol} and Proposition \ref{idbundle01}, for each $\varepsilon_i$, we have
\begin{equation} \label{seq1}
\int_X \textmd{tr}(\Phi(K,\theta)u_{\varepsilon_i})\frac{\omega^n}{n!}+l_i\int_{X}\langle \Psi(l_iu_{\varepsilon_i})(\overline{\partial}_{\theta}u_{\varepsilon_i}),\overline{\partial}_{\theta}u_{\varepsilon_i}\rangle_{K}\frac{\omega^n}{n!}=-\varepsilon_il_i.
\end{equation}
Consider the function
\begin{equation*}
l\Psi(lx,ly)=
\begin{cases}
\ \ \ \ l,\ \ &x=y;\\
\frac{e^{l(y-x)}-1}{y-x},\ \ &x\neq y.
\end{cases}
\end{equation*}
From  \eqref{uc0}, we may assume that $(x,y)\in [-C_{10},C_{10}]\times[-C_{10},C_{10}]$. It is easy to check that
\begin{equation} \label{seq2}
l\Psi(lx,ly)\rightarrow
\begin{cases}
(x-y)^{-1},\ \ \ &x>y;\\
\ \ +\infty,\ \ \ &x\leq y,
\end{cases}
\end{equation}
increases monotonically as $l\rightarrow +\infty$.  Let $\varsigma\in C^{\infty} (\mathbb{R}\times \mathbb{R}, \mathbb{R}^+)$ satisfying $\varsigma(x,y)<(x-y)^{-1}$ whenever $x>y$. From Eqs. \eqref{seq1}, \eqref{seq2}  and the arguments in \cite[Lemma 5.4]{SIM}, we have
\begin{equation} \label{seq3}
\int_X \textmd{tr}(\Phi(K,\theta)u_{\varepsilon_i})\frac{\omega^n}{n!}+\int_{X}\langle \varsigma(u_{\varepsilon_i})(\overline{\partial}_{\theta}u_{\varepsilon_i}),\overline{\partial}_{\theta}u_{\varepsilon_i}\rangle_{K}\frac{\omega^n}{n!}\leq 0, \ \ i\gg 0.
\end{equation}
In particular, we take $\zeta(x,y)=\frac{1}{3C_{10}}.$ It is obvious that when  $(x,y)\in [-C_{10},C_{10}]\times[-C_{10},C_{10}]$ and $x>y$, $\frac{1}{3C_2}<\frac{1}{x-y}$. This implies that
$$
\int_X\textmd{tr}(\Phi(K,\theta)u_{\varepsilon_i})\frac{\omega^n}{n!}
+\frac{1}{3C_{10}}\int_{X}|\overline{\partial}_{\theta}(u_{\varepsilon_i})|^2_{K}\frac{\omega^n}{n!}\leq 0,
$$
for $i\gg 0$. Then we have
$$
\int_{X}|\overline{\partial}_{\theta}(u_{\varepsilon_i})|^2_{K}\frac{\omega^n}{n!}\leq 3C^2_{10}\sup\limits_X|\Phi(K,\theta)|_{K}\textmd{Vol}(X).
$$
Thus, $u_{\varepsilon_i}$ are bounded in $L_1^2$. We can choose a subsequence $\{u_{\varepsilon_{i_j}}\}$ such that $u_{\varepsilon_{i_j}}\rightharpoonup u_{\infty}$ weakly in $L^2_1,$  still denoted by $\{u_{\varepsilon_i}\}^{\infty}_{i=1}$ for simplicity. Noting that $L_1^2\hookrightarrow L^2,$ we have
 $$1=\int_X | u_{\varepsilon_i}|^2_{K}\rightarrow \int_X | u_{\infty}|^2_{K}.$$
 This indicates that $\| u_{\infty}\|_{L^2}=1$ and $u_{\infty}$ is non-trivial.
Using  \eqref{seq3} and following a similar discussion as in \cite[Lemma 5.4]{SIM}, we have
\begin{equation} \label{seq4}
\int_X \textmd{tr}(\Phi(K,\theta)u_{\infty})\frac{\omega^n}{n!}+\int_{X}\langle \varsigma(u_{\infty})(\overline{\partial}_{\theta}u_{\infty}),\overline{\partial}_{\theta}u_{\infty}\rangle_{K}\frac{\omega^n}{n!}\leq 0.
\end{equation}

$\bullet$ $Step \ 2$ Using Uhlenbeck and Yau's trick (\cite{UY86}) and Simpson's argument (\cite{SIM}), we construct a Higgs subsheaf which contradicts the stability of $(E,\bar{\partial}_E,\theta)$.

By  \eqref{seq4} and the same argument in \cite[Lemma 5.5]{SIM}, we conclude that the eigenvalues of $u_{\infty}$ are constant almost everywhere. Let $\mu_1<\mu_2<\cdots<\mu_l$ be the distinct eigenvalues of $u_{\infty}$. The facts that $\textmd{tr}(u_{\infty})=0$ and $\|u_{\infty}\|_{L^2}=1$ force $2\leq l\leq r$. For each $\mu_{\alpha} (1\leq \alpha\leq l-1)$, we construct a function $P_{\alpha}: \mathbb{R}\rightarrow \mathbb{R}$ such that
$$
P_{\alpha}=
\begin{cases}
1,\ \ \ x\leq \mu_{\alpha};\\
0,\ \ \ x\geq \mu_{\alpha+1}.
\end{cases}
$$
Setting $\pi_{\alpha}=P_{\alpha}(u_{\infty})$, from \cite[p.887]{SIM}, we have: (i) $\pi_{\alpha}\in L^2_1$; (ii)$\pi_{\alpha}^2=\pi_{\alpha}=\pi_{\alpha}^{*_{K}}$; (iii) $(\textmd{Id}_E-\pi_{\alpha})\bar{\partial}\pi_{\alpha}=0$ and (iv) $(\textmd{Id}_E-\pi_{\alpha})[\theta, \pi_{\alpha}]=0$. By Uhlenbeck and Yau's regularity statement of $L^2_1$-subbundle \cite{UY86}, $\{\pi_{\alpha}\}_{\alpha=1}^{l-1}$ determine $l-1$ Higgs sub-sheaves of $E$. Set $E_{\alpha}=\pi_{\alpha}(E)$. From $\textmd{tr}(u_{\infty})=0$ and $u_{\infty}=\mu_l \cdot \textmd{Id}_E-\sum\limits_{\alpha=1}^{l-1}(\mu_{\alpha+1}-\mu_{\alpha})\pi_{\alpha}$, it holds
\begin{equation} \label{seq5}
\mu_{l}\textmd{rank}(E)=\sum_{\alpha=1}^{l-1}(\mu_{\alpha+1}-\mu_{\alpha})\textmd{rank}(E_{\alpha}).
\end{equation}
Construct
$$\nu=\mu_l \textmd{deg}(E,K)-\sum_{\alpha=1}^{l-1}(\mu_{\alpha+1}-\mu_{\alpha})\textmd{deg}(E_{\alpha},K).$$
Substituting Eq. \eqref{seq5} into $\nu$,
\begin{equation} \label{seq6}
\nu=\sum_{\alpha=1}^{l-1}(\mu_{\alpha+1}-\mu_{\alpha})\textmd{rank}(E_{\alpha})(\frac{\textmd{deg}(E,K)}{\textmd{rank}(E)}
-\frac{\textmd{deg}(E_{\alpha},K)}{\textmd{rank}(E_{\alpha})}).
\end{equation}
On the other hand, substituting Eq. \eqref{cw} into $\nu$ we have
\begin{align}
\nu&=\mu_l\int_X \sqrt{-1}\textmd{tr}(\Lambda_{\omega } F_{K,\theta})
-\sum_{\alpha=1}^{l-1}(\mu_{\alpha+1}-\mu_{\alpha})\{\int_X \sqrt{-1}\textmd{tr}(\pi_{\alpha}\Lambda_{\omega } F_{K,\theta})-|\overline{\partial}_{\theta}\pi_{\alpha}|^2\}\notag\\
&=\int_X \textmd{tr}\{(\mu_l\cdot \textmd{Id}_E-\sum_{\alpha=1}^{l-1}(\mu_{\alpha+1}-\mu_{\alpha})\pi_{\alpha})(\sqrt{-1}\Lambda_{\omega } F_{K,\theta})\}+\sum_{\alpha=1}^{l-1}(\mu_{\alpha+1}-\mu_{\alpha})\int_X|\overline{\partial}_{\theta}\pi_{\alpha}|^2\notag\\
&=\int_X \textmd{tr}(u_{\infty}\sqrt{-1}\Lambda_{\omega } F_{K,\theta})
+\langle \sum_{\alpha=1}^{l-1}(\mu_{\alpha+1}-\mu_{\alpha})(\textmd{d} P_{\alpha})^2(u_{\infty})(\overline{\partial}_{\theta}u_{\infty}),\overline{\partial}_{\theta}u_{\infty}\rangle_{K},\notag
\end{align}
where the function $\textmd{d} P_{\alpha}: \mathbb{R}\times\mathbb{R}\rightarrow \mathbb{R}$ is defined by
$$
\textmd{d} P_{\alpha}(x,y)=
\begin{cases}
\frac{P_{\alpha}(x)-P_{\alpha}(y)}{x-y}, \ \ x\neq y;\\
P'_{\alpha}(x), \ \ x=y.
\end{cases}
$$
One can easily check that,
\begin{equation*}
\sum_{\alpha=1}^{l-1}(\mu_{\alpha+1}-\mu_{\alpha})(\textmd{d} P_{\alpha})^2(\mu_{\beta},\mu_{\gamma})=|\mu_{\beta}-\mu_{\gamma}|^{-1},
\end{equation*}
if $\mu_{\beta}\neq \mu_{\gamma}$. Then using \eqref{seq4}, we have
\begin{equation} \label{seq7}
\begin{split}
\nu&=\int_X\textmd{tr}(u_{\infty}\sqrt{-1}\Lambda_{\omega } F_{K,\theta})
+\langle \sum_{\alpha=1}^{l-1}(\mu_{\alpha+1}-\mu_{\alpha})(\textmd{d} P_{\alpha})^2(u_{\infty})(\overline{\partial}_{\theta}u_{\infty}),\overline{\partial}_{\theta}u_{\infty}\rangle_{K}\\
&\leq 0.
\end{split}
\end{equation}
Combining \eqref{seq6} and \eqref{seq7}, we have
$$\sum_{\alpha=1}^{l-1}(\mu_{\alpha+1}-\mu_{\alpha})\textmd{rank}(E_{\alpha})(\frac{\textmd{deg}(E,K)}{\textmd{rank}(E)}
-\frac{\textmd{deg}(E_{\alpha},K)}{\textmd{rank}(E_{\alpha})})\leq 0,$$
which contradicts the stability of $E$.

\hfill $\Box$ \\

In the following, we will prove that the semi-stability implies the existence of approximate Hermitian-Einstein structure.

\medskip

{\bf Proof of Theorem \ref{theorem2}}

We only need to prove the following claim.

\medskip

\textbf{Claim}  If $(E,\bar{\partial}_E, \theta)$ is semi-stable, then it holds
$$
\lim\limits_{\varepsilon\rightarrow 0}\sup\limits_X |\sqrt{-1}\Lambda_{\omega } F_{H_{\varepsilon},\theta}-\lambda_{K,\omega} \cdot \textmd{Id}_E|_{H_{\varepsilon}}=\lim\limits_{\varepsilon\rightarrow 0}\varepsilon \sup \limits_X | \log h_{\varepsilon}|_{H_{\varepsilon}}=0.
$$

\begin{proof}

If the claim does not hold, then there exist $\delta>0$ and a subsequence $\varepsilon_i\rightarrow 0, i\rightarrow+\infty$, such that
\begin{equation} \label{claim1}
\sup\limits_X|\sqrt{-1}\Lambda_{\omega } F_{H_{\varepsilon_i},\theta}-\lambda_{K,\omega} \cdot \textmd{Id}_E|_{H_{\varepsilon_i}}=\varepsilon_i\sup_X|\log h_{\varepsilon_i}|_{H_{\varepsilon_i}}\geq \delta,
\end{equation}
for any $\varepsilon_{i}$, and
\begin{equation*}
\| \log h_{\varepsilon_i}\|_{L^2}\rightarrow +\infty.
\end{equation*}
Setting
$$s_{\varepsilon_i}=\log h_{\varepsilon_i}, \ \ l_i=\| s_{\varepsilon_i}\|_{L^2}, \ \ u_{\varepsilon_i}=\frac{s_{\varepsilon_i}}{l_i},$$ we have
$$\textmd{tr} (u_{\varepsilon_i})=0, \ \| u_{\varepsilon_i}\|_{L^2}=1.$$
By \eqref{claim1} and Lemma \ref{lastsectionlemma}, we have
\begin{equation} \label{claim2}
l_i\geq \frac{\delta}{\varepsilon_iC_7}-\frac{C_8}{C_7}
\end{equation}
and
\begin{equation*} \label{claim3}
\sup\limits_X|u_{\varepsilon_i}|\leq \frac{1}{l_i}(C_7l_i+C_8)<C_{10}<+\infty.
\end{equation*}
By (\ref{seq1}) and \eqref{claim2}, we have
\begin{equation} \label{claim4}
\frac{\delta}{C_7}+\int_X \textmd{tr}(\Phi(K,\theta)u_{\varepsilon_i})\frac{\omega^n}{n!}+l_i\int_{X}\langle \Psi(l_iu_{\varepsilon_i})(\overline{\partial}_{\theta}u_{\varepsilon_i}),\overline{\partial}_{\theta}u_{\varepsilon_i}\rangle_{K}\frac{\omega^n}{n!}\leq \varepsilon_i\frac{C_8}{C_7}.
\end{equation}
By \eqref{claim4}  and the arguments in \cite[Lemma 5.4]{SIM}, we have
\begin{equation} \label{claim6}
\frac{\delta}{2C_7}+\int_X \textmd{tr}(\Phi(K,\theta)u_{\varepsilon_i})\frac{\omega^n}{n!}+\int_{X}\langle \varsigma(u_{\varepsilon_i})(\overline{\partial}_{\theta}u_{\varepsilon_i}),\overline{\partial}_{\theta}u_{\varepsilon_i}\rangle_{K}\frac{\omega^n}{n!}\leq 0, \ \ i\gg 0.
\end{equation}

By the same argument as that in $Step \ 1$ in the proof of Theorem \ref{theorem1}, we can prove that  $\|\overline{\partial}_{\theta}u_{\varepsilon_i}\|_{L^2}$ are uniformly bounded.
By choosing a subsequence, we have  $u_{\varepsilon_{i}}\rightharpoonup u_{\infty}$ weakly in $L^2_1$,  and $\| u_{\infty}\|_{L^2}=1$. Using Eq. \eqref{claim6} and following a similar discussion as in \cite[Lemma 5.4, Lemma 5.5]{SIM}, we have
\begin{equation} \label{claim7}
\frac{\delta}{2C_7}+\int_X \textmd{tr}(\Phi(K,\theta)u_{\infty})\frac{\omega^n}{n!}+\int_{X}\langle \varsigma(u_{\infty})(\overline{\partial}_{\theta}u_{\infty}),\overline{\partial}_{\theta}u_{\infty}\rangle_{K}\frac{\omega^n}{n!}\leq 0.
\end{equation}
and $u_{\infty}=\mu_l\cdot \textmd{Id}_E-\sum\limits_{\alpha=1}^{l-1}(\mu_{\alpha+1}-\mu_{\alpha})\pi_{\alpha}$, where $\mu_1<\mu_2<\cdots<\mu_l$, $\{\pi_{\alpha}\}_{\alpha=1}^{l-1}$ determine $l-1$ Higgs sub-sheaves $\{E_{\alpha}\}_{\alpha=1}^{l-1}:=\{\pi_{\alpha}(E)\}_{\alpha=1}^{l-1}$ of $E$.

 By  \eqref{claim7} and the same arguments in \cite[p.793-794]{LZ}, we have
\begin{equation*} \label{uyr5}
\begin{split}
\nu &=\sum_{\alpha=1}^{l-1}(\mu_{\alpha+1}-\mu_{\alpha})\textmd{rank}(E_{\alpha})(\frac{\textmd{deg}(E,K)}{\textmd{rank}(E)}
-\frac{\textmd{deg}(E_{\alpha},K)}{\textmd{rank}(E_{\alpha})})\\
&=\int_X\textmd{tr}(u_{\infty}\sqrt{-1}\Lambda_{\omega } F_{K,\bar{\partial}_E,\theta})
+\langle \sum_{\alpha=1}^{l-1}(\mu_{\alpha+1}-\mu_{\alpha})(\textmd{d} P_{\alpha})^2(u_{\infty})(\overline{\partial}_{\theta}u_{\infty}),\overline{\partial}_{\theta}u_{\infty}\rangle_{K}\\
&\leq -\frac{\delta}{2C_7},
\end{split}
\end{equation*}
which contradicts the semi-stability of $(E,\bar{\partial}_E, \theta)$. This completes the proof of the claim.
\end{proof}

\vskip 1 true cm


\bigskip
\bigskip

\noindent {\footnotesize {\it Chuanjing Zhang, Pan Zhang and Xi Zhang} \\
{School of Mathematical Sciences, University of Science and Technology of China}\\
{Anhui 230026, P.R. China}\\
{Email:chjzhang@mail.ustc.edu.cn; panzhang@mail.ustc.edu.cn; mathzx@ustc.edu.cn}

\vskip 0.5 true cm


\begin{thebibliography}{20}



\bibitem{ag}
L. Alvarez-Consul and O. Garcis-Prada, Hitchin-Kobayashi correspondence, quivers, and vortices, Commun. Math. Phys. {\bf238}(2003), 1-33.

\bibitem{Bi}
O. Biquard,  On parabolic bundles over a complex surface, J. London Math. Soc. {\bf 53}(1996), 302-316.

\bibitem{bis}
I. Biswas and G. Schumacher, Yang-Mills equation for stable Higgs sheaves, Inter. J. Math. {\bf20}(2009), 541-556.

\bibitem{br}
S.B. Bradlow, Vortices in holomorphic line bundles over closed K\"{a}hler manifolds, Commun. Math. Phys.
{\bf135}(1990), 1-17.

\bibitem{bruzzo} U. Bruzzo and B. Gra\~na Otero,
Metrics on semistable and numerically effective Higgs bundles,
 J. Reine Angew. Math. {\bf 612}(2007), 59-79.

\bibitem{Bu}
N.P. Buchdahl,
 Hermitian-Einstein connections and stable vector bundles over compact complex surfaces,
  Math. Ann. {\bf 280}(1988),  625--648.

\bibitem{Cardona}
S.A.H. Cardona, Approximate Hermitian-Yang-Mills structures and semistability for Higgs bundles I: generalities and the one-dimensional case,
 Ann. Global Anal. Geom. {\bf 42}(2012), 349-370.

\bibitem{DW}
G. Daskalopoulos and R. Wentworth, Convergence properties of the Yang-Mills flow on K\"{a}hler surfaces. J. Reine Angew. Math. {\bf575}(2004), 69-99.

\bibitem{DON85}
S.K. Donaldson, Anti self-dual Yang-Mills connections over complex algebraic
surfaces and stable vector bundles, Proc. London Math. Soc. {\bf50}(1985), 1-26.

\bibitem{DON92}
S.K. Donaldson, Boundary value problems for Yang-Mills fields, J. Geom. Phy. {\bf8}(1992), 89-122.


\bibitem{HAMILTON}
R.S. Hamilton, Harmonic maps of manifolds with boundary, Lecture Notes in Math., Vol. 471, Springer, New York, 1975.

\bibitem{HIT}
N.J. Hitchin, The self-duality equations on a Riemann surface, Proc. London Math. Soc. {\bf55}(1987), 59¨C126.

\bibitem{Jacob1}
A. Jacob, Existence of approximate Hermitian-Einstein structures on semi-stable bundles,
 Asian J. Math. {\bf18}(2014), 859-883.

\bibitem{JZ}
J. Jost and K. Zuo, Harmonic maps and $Sl(r,C)$-representations of fundamental groups of quasiprojective manifolds, J. Algebraic Geom. {\bf 5}(1996), 77-106.

\bibitem{Kobayashi}
S. Kobayashi, Differential geometry of complex vector bundles,  Publications of the Mathematical Society of Japan, 15, Princeton University Press, Princeton, NJ, 1987.


\bibitem{LN2} J.Y. Li and M.S. Narasimhan,
Hermitian-Einstein metrics on parabolic stable bundles,
 Acta Math. Sin. (Engl. Ser.) {\bf 15}(1999), 93-114.

\bibitem{LZ}
J.Y. Li and X. Zhang, Existence of approximate Hermitian-Einstein structures on semi-stable Higgs bundles, Calc. Var. {\bf52}(2015), 783-795.


\bibitem{LZZ}
J.Y. Li, C. Zhang and X. Zhang, Semi-stable Higgs sheaves and Bogomolov type inequality, Calc. Var. {\bf56}(2017), 1-33.

\bibitem{LZZ2}
J.Y. Li, C. Zhang and X. Zhang, The limit of the Hermitian-Yang-Mills flow on reflexive sheaves, Adv. Math.  {\bf325}(2018), 165-214.

\bibitem{LY}
J. Li and S.T. Yau, Hermitian-Yang-Mills connection on non-K\"ahler manifolds,
 Mathematical aspects of string theory (San Diego, Calif., 1986), 560-573, Adv. Ser. Math. Phys., 1, World Sci. Publishing, Singapore, 1987.

\bibitem{LT}
M. L\"{u}bke and A. Teleman, The universal Kobayashi-Hitchin correspondence on Hermitian manifolds, Mem. Amer. Math. Soc., 2006.


\bibitem{LT95}
M. L\"{u}bke and A. Teleman, The Kobayashi-Hitchin correspondence, World
Scientific Publishing Co., Inc., River Edge, NJ, 1995.


\bibitem{Mo1}
 T. Mochizuki, Kobayashi-Hitchin correspondence for tame harmonic bundles and an application,
Ast\'erisque, {\bf 309}, Soc. Math. France, Paris, 2006.


\bibitem{Mo2}
T. Mochizuki,  Kobayashi-Hitchin correspondence for tame harmonic bundles II, Geom. Topol. {\bf 13}(2009), 359-455.

\bibitem{Mo3}
T. Mochizuki, Kobayashi-Hitchin correspondence for analytically stable bundles, arXiv: 1712.08978v1.

\bibitem{NZ}
Y. Nie, X. Zhang, Semistable Higgs bundles over compact Gauduchon manifolds, J. Geom. Anal. {\bf28}(2018), 627-642.


\bibitem{NR01}
L. Ni and H. Ren, Hermitian-Einstein metrics for vector bundles on complete K\"{a}hler manifolds,  Trans. Amer. Math. Soc. {\bf353}(2001), 441-456.

\bibitem{NS65}
M.S. Narasimhan and C.S. Seshadri, Stable and unitary vector bundles
on a compact Riemann surface, Ann. Math. {\bf82}(1965), 540-567.



\bibitem{SIM}
C.T. Simpson, Constructing variations of Hodge structure using Yang-Mills theory and applications to uniformization, J. Amer. Math. Soc. {\bf1}(1988), 867-918.

\bibitem{SIM2}
C.T. Simpson, Higgs bundles and local systems, Inst. Hautes \'{E}tudes Sci. Publ. Math. {\bf75}(1992), 5-95.

\bibitem{T}
M.E. Taylor, Partial differential equations I, Applied Mathematical Sciences, Vol. 115, Springer-Verlag, New York, Berlin, Heidelberg.

\bibitem{UY86}
K.K. Uhlenbeck and S.-T. Yau, On the existence of Hermitian-Yang-Mills
connections in stable vector bundles, Comm. Pure Appl. Math. {\bf39S}(1986),
S257-S293.

\bibitem{WZ}
Y. Wang and X. Zhang, Twisted holomorphic chains and vortex equations over non-compact K\"{a}hler manifolds, J. Math. Anal. Appl. {\bf373}(2011), 179-202.

\bibitem{Z}
X. Zhang, Hermitian-Einstein metrics on holomorphic vector bundles over Hermitian manifolds, J. Geom. Phys. {\bf53}(2005), 315-335.


\end{thebibliography}
\end{document}